%% file: 0_article.tex
\documentclass[11pt]{article}

\usepackage{microtype}
\usepackage{graphicx}
\usepackage{subfigure}
\usepackage{booktabs} 

\usepackage{hyperref}

\usepackage{amsmath}
\usepackage{amssymb}
\usepackage{mathtools}
\usepackage{amsthm}

\usepackage[capitalize,noabbrev]{cleveref}

\input{5_macros}

\input{6_ColorBox}
\theoremstyle{plain}
\newtheorem{theorem}{Theorem}[section]

\newtheorem{lemma}[theorem]{Lemma}

\theoremstyle{definition}

\newtheorem{assumption}[theorem]{Assumption}
\theoremstyle{remark}
\newtheorem{remark}[theorem]{Remark}

\title{\bf Linear Convergence Rate in Convex Setup is Possible! \\Gradient Descent Method Variants under $(L_0,L_1)$-Smoothness}

\author{\begin{tabular}{c}
     Aleksandr Lobanov\footnote{Part of the work was done while A.~Lobanov was an intern at MBZUAI.} \\ 
     MIPT, Skoltech, ISP RAS\\
     \texttt{lobbsasha@mail.ru}
\end{tabular} \and \begin{tabular}{c}
     Alexander Gasnikov \\ 
     Innopolis University, MIPT, ISP RAS\\
     \texttt{gasnikov@yandex.ru}
\end{tabular} \and \begin{tabular}{c}
     Eduard Gorbunov\\
     MBZUAI\\
     \texttt{eduard.gorbunov@mbzuai.ac.ae}
\end{tabular} \and  \begin{tabular}{c}
     Martin Tak\'a\v{c}\\
     MBZUAI\\
     \texttt{martin.takac@mbzuai.ac.ae}
\end{tabular}}

\date{\today\footnote{{\bf The first version was submitted to arXiv on December 22, 2024. The second version contains the following changes:} 1) minor inaccuracies in the proofs from Section~\ref{sec:Full-Gradient Methods} were fixed, the results remained the same; 2) major inaccuracies in the proofs from Section~\ref{sec:Coordinate Descent Type Methods} were fixed, and the results were simplified to the uniform sampling case; 3) major inaccuracies in the results from Section~\ref{sec:Extension to Strongly Convex Setup} were fixed; 4) writing was improved, missing references were added (including the updated version of \citep{Vankov_2024_ver_2}).}}

\begin{document}

\maketitle

\input{1_abstract}

\tableofcontents

\input{2_main}

\bibliography{3_references}
\bibliographystyle{apalike}

\input{4_appendix}

\end{document}

%% file: 5_macros.tex
\usepackage{fullpage}

\usepackage{amsfonts}
\usepackage{dsfont}
\usepackage{wrapfig}

\usepackage{natbib}
\usepackage{sidecap}

\usepackage{algorithm}
\allowdisplaybreaks

\usepackage{algpseudocode}
\usepackage{mdframed}

\usepackage{xcolor}
\usepackage{colortbl}
\usepackage{color}
\usepackage{graphicx}
\usepackage{multirow}
\usepackage{pifont}

\usepackage{bbm}

\usepackage{nicefrac}

\usepackage[font=small]{caption}
\usepackage[font=small]{subcaption}

\usepackage{verbatim}
\usepackage{xspace} %

\usepackage{enumerate}
\usepackage{enumitem}

\newcommand{\dotprod}[2]{\left\langle #1,#2 \right\rangle}
\newcommand{\Obound}[1]{\mathcal{O}\left( #1 \right)}

\newcommand{\norms}[1]{\left\| #1 \right\| }

\definecolor{PineGreen}{HTML}{008B72}
\newcommand{\greencheck}{\color{PineGreen}\ding{51}}
\newcommand{\redx}{\color{red} \ding{55}}

  \providecommand{\R}{\mathbb{R}} %

  \DeclareMathOperator{\E}{{\mathbb E}}

  \DeclareMathOperator*{\argmin}{arg\,min}
  \DeclareMathOperator*{\arginf}{arg\,inf}

  \providecommand{\ee}{\mathbf{e}}

  \providecommand{\mL}{\mathbf{L}}

  \providecommand{\cK}{\mathcal{K}}

  \providecommand{\cO}{\mathcal{O}}

  \providecommand{\cT}{\mathcal{T}}

  \usepackage{bm}

\usepackage[textsize=tiny]{todonotes}

\providecommand{\mycomment}[3]{\todo[caption={},size=footnotesize,color=#1!20, inline]{\textbf{#2: }#3}}%
\providecommand{\inlinecomment}[3]{%
  {\color{#1}#2: #3}}%
\newcommand\commenter[2]%
{%
  \expandafter\newcommand\csname i#1\endcsname[1]{\inlinecomment{#2}{#1}{##1}}
  \expandafter\newcommand\csname #1\endcsname[1]{\mycomment{#2}{#1}{##1}}
}

\newcommand{\circledOne}{\text{\ding{172}}}
\newcommand{\circledTwo}{\text{\ding{173}}}
\newcommand{\circledThree}{\text{\ding{174}}}

\commenter{AL}{blue}

\usepackage{url}

%% file: 6_ColorBox.tex
\usepackage[many]{tcolorbox}    	

\definecolor{main}{HTML}{5989cf}    
\definecolor{sub}{HTML}{cde4ff}     

\tcbset{
    sharp corners,
    colback = white,
    before skip = 0.2cm,    
    after skip = 0.5cm      
}                           

\newtcolorbox{boxA}{
    colback = white,
    enhanced,
    boxrule = 1.5pt, 
    colframe = black, 
    borderline = {1.5pt}{0pt}{black, dashed} 
}

\newtcolorbox{boxF}{
    colback = white,
    enhanced,
    boxrule = 1.5pt, 
    colframe = white, 
    borderline = {1.5pt}{0pt}{black, dashed} 
}

\newtcolorbox{boxB}{
    fontupper = \bf\color{main}, 
    boxrule = 1.5pt,
    colframe = main,
    rounded corners,
    arc = 5pt   
}

\newtcolorbox{boxC}{
    colback = sub, 
    boxrule = 0pt  
}

\newtcolorbox{boxD}{
    colback = white, 
    colframe = black, 
    boxrule = 0pt, 
    toprule = 3pt, 
    bottomrule = 3pt 
}

\newtcolorbox{boxE}{
    enhanced, 
    boxrule = 0pt, 
    borderline = {0.75pt}{0pt}{main}, 
    borderline = {0.75pt}{2pt}{sub} 
}

\newtcolorbox{boxG}{
    enhanced,
    boxrule = 0pt,
    colback = sub,
    borderline west = {1pt}{0pt}{main}, 
    borderline west = {0.75pt}{2pt}{main}, 
    borderline east = {1pt}{0pt}{main}, 
    borderline east = {0.75pt}{2pt}{main}
}

\newtcolorbox{boxH}{
    colback = white, 
    colframe = black, 
    boxrule = 0pt, 
    leftrule = 6pt 
}

\newtcolorbox{boxI}{
    colback = sub, 
    colframe = main, 
    boxrule = 0pt, 
    toprule = 6pt 
}

\newtcolorbox{boxJ}{
    sharpish corners, 
    colback = sub, 
    colframe = main, 
    boxrule = 0pt, 
    toprule = 4.5pt, 
    enhanced,
    fuzzy shadow = {0pt}{-2pt}{-0.5pt}{0.5pt}{black!35} 
}

\newtcolorbox{boxK}{
    sharpish corners, 
    boxrule = 0pt,
    toprule = 4.5pt, 
    enhanced,
    fuzzy shadow = {0pt}{-2pt}{-0.5pt}{0.5pt}{black!35} 
}

\newtcolorbox{boxL}{
    fontupper = \color{main},
    rounded corners,
    arc = 6pt,
    colback = sub, 
    colframe = main!50, 
    boxrule = 0pt, 
    bottomrule = 4.5pt 
}

\newtcolorbox{boxM}{
    fontupper = \color{white},
    rounded corners,
    arc = 6pt,
    colback = main!80, 
    colframe = main, 
    boxrule = 0pt, 
    bottomrule = 4.5pt,
    enhanced,
    fuzzy shadow = {0pt}{-3pt}{-0.5pt}{0.5pt}{black!35}
}

%% file: 1_abstract.tex
\begin{abstract}
The gradient descent (GD) method -- is a fundamental and likely the most popular optimization algorithm in machine learning (ML), with a history traced back to a paper in 1847 \cite{Cauchy_1847}. It was studied under various assumptions, including so-called $(L_0,L_1)$-smoothness, which received noticeable attention in the ML community recently. In this paper, we provide a refined convergence analysis of gradient descent and its variants, assuming generalized smoothness. In particular, we show that $(L_0,L_1)$-GD 
has the following behavior in the \textit{convex setup}: as long as $\norms{\nabla f(x^k)} \geq \frac{L_0}{L_1}$ the algorithm has \textit{linear convergence} in function suboptimality, and when $\norms{\nabla f(x^k)} < \frac{L_0}{L_1}$ is satisfied, $(L_0,L_1)$-GD has standard sublinear rate. Moreover, we also show that this behavior is common for its variants with different types of oracle: \textit{Normalized Gradient Descent} as well as \textit{Clipped Gradient Descent} (the case when the full gradient $\nabla f(x)$ is available); \textit{Random Coordinate Descent} (when the gradient component $\nabla_{i} f(x)$ is available); \textit{Random Coordinate Descent with Order Oracle} (when only $\text{sign} [f(y) - f(x)]$ is available). In addition, we also extend our analysis of $(L_0,L_1)$-GD to the strongly convex case. 
\end{abstract}

%% file: 2_main.tex
\section{Introduction}
We consider the standard unconstrained minimization
\begin{equation} \label{eq:init_problem}
    \min_{x \in \mathbb{R}^d} f(x),
\end{equation}
where $f: \mathbb{R}^d \rightarrow \mathbb{R}$ is a convex differentiable function. This problem configuration is quite general and encompasses a broad range of applications in ML scenarios. For such problems, the traditional optimization algorithm is the \textit{gradient descent method} (GD) \cite{Cauchy_1847}, which has a sublinear convergence rate in the convex setting under the Lipschitz smoothness assumption \citep[see, e.g.,][]{Nesterov_2013}. In particular, GD is the core of optimization for machine learning, and various modifications of this method have been studied in different assumptions suited to ML applications.

In this paper, we consider one of such assumptions called \emph{$(L_0,L_1)$-smoothness} \citep{Zhang_2019, Zhang_2020, Chen_2023}, which in the case of twice differentiable functions, states that $\|\nabla^2 f(x)\| \leq L_0 + L_1\|\nabla f(x)\|$, i.e., the smoothness constant can grow as a linear function of the gradient norm. Under this assumption, different variants of GD are analyzed, including 
GD with clipping (Clip-GD) \citep{Zhang_2019, Zhang_2020, Koloskova_2023, Vankov_2024}, $(L_0,L_1)$-GD \citep{Gorbunov_2024, Vankov_2024}, Normalized GD (NGD) \citep{zhao2021convergence, Chen_2023, Vankov_2024}, and other variants \citep{crawshaw2022robustness, wang2022provable, faw2023beyond, wang2023convergence, hubler2024parameter, li2024convergence}. More precisely, in the deterministic convex case, the state-of-the-art results for Clip-GD, $(L_0,L_1)$-GD, and NGD are obtained by \citet{Gorbunov_2024, Vankov_2024} showing the $\cO\left(\frac{L_0R^2}{N}\right)$ rates for function suboptimality \underline{when $N = \Omega(L_1^2R^2)$}\footnote{
After the first version of our work appeared on arXiv, \citet{Vankov_2024} let us know that they also independently derived $\cO\left(\frac{L_0R^2}{N} + \left(1 - \frac{1}{L_1R}\right)^N F_0\right)$ rate for $(L_0,L_1)$-GD, where $F_0 = f(x^0) - f(x^*)$ during the discussion with reviewers of their work to a ML conference. Then, the authors updated their paper on arXiv \citep{Vankov_2024_ver_2}. At the moment of writing our paper, we were unaware of the updated version of \citep{Vankov_2024_ver_2}.
} leaving open questions about the refined methods behavior characterization for $N = \cO(L_1^2 R^2)$.

However, beyond the first-order methods, the algorithms for $(L_0,L_1)$-smooth optimization are weakly studied. In particular, \emph{random coordinate descent} (RCD) \cite{Nesterov_2012,Shalev_2009,Richtarik_2016}, which is useful in the situations when the computation of the full gradient is prohibitively expensive, is not analyzed in the context of $(L_0,L_1)$-smooth optimization. Moreover, in some cases, e.g., in the reinforcement learning with human feedback \citep{Tang_2024}, even objective values are available, and for given points $x,y\in \R^d$ one can only evaluate $\text{sign}[f(y) - f(x)]$. To the best of our knowledge, there are no theoretical convergence results for such methods under $(L_0,L_1)$-smoothness, and, in particular, the convergence of \emph{random coordinate descent with order oracle} (OrderRCD) \citep{Lobanov_2024} is not studied in this setup.

In this paper, we address this gap in the literature and provide the first analysis of RCD and OrderRCD for convex $(L_0,L_1)$-smooth optimization. Moreover, we improve the existing results for $(L_0,L_1)$-GD, NGD, and Clip-GD: we prove that these methods enjoy \emph{linear convergence rates without any additional assumptions} for the initial optimization phase when $\norms{\nabla f(x^k)} \geq \frac{L_0}{L_1}$.

Our contributions can be summarized as follows.
\begin{itemize}
    \item  We show under what conditions variants of the gradient descent achieve linear convergence in the convex setup.

    \item  We prove better complexity bounds for $(L_0,L_1)$-GD, NGD, and Clip-GD than previously known ones, assuming convexity and $(L_0,L_1)$-smoothness of the objective function. We show that these algorithms converge linearly at first, and slow down as they approach the solution, converging sublinearly. We also show that for the phase of convergence of NGD, when the iterates satisfy $\norms{\nabla f(x^k)} \geq c$, the method converges linearly. Table~\ref{tab:table_compare} demonstrates the conditions under which Clip-GD converges linearly. In particular, the case of $\lambda_k = 1$ corresponds to the convergence of GD, and the case of $\lambda_k = \frac{c}{\norms{\nabla f(x^k)}}$ corresponds to the convergence of NGD. 

    \item We provide the first convergence results for the RCD and OrderRCD algorithms under the convexity and $(L_0,L_1)$-coordinate smoothness assumptions. We demonstrate that the linear convergence phenomenon of the full-gradient methods exists for both of the mentioned methods.
    
    \item We extend our analysis of $(L_0,L_1)$-GD to the case when the function is $\mu$-strongly convex.
\end{itemize}

\begin{table*}
    \begin{minipage}{\textwidth}
        \caption{Comparison of the convergence rates for Clip-GD in the convex case. Clip-GD update scheme: $x^{k+1} = x^k - \eta_k \cdot \text{clip}_c(\nabla f(x^k))$. Notation: $\text{clip}_c(\nabla f(x^k)) = \lambda_k \cdot \nabla f(x^k)$; $\lambda_k = \min \{ 1, \nicefrac{c}{\norms{\nabla f(x^k)}}\}$; $c>0$ -- clipping radius; $\eta_k > 0$ -- step size; $N=$ number of iterations; $F_0 = f(x^0) - f^*$; $R = \norms{x^0 - x^*}$; $T = \min \left\{ k \in \{0,1,...,N-1\}\; | \;\|\nabla f(x^k)\| < L_0/L_1 \right\}$; LCR~$=$~linear~convergence~rate.}
    \label{tab:table_compare} 
    \centering
    \resizebox{\linewidth}{!}{ 
    \begin{tabular}{lccclc}\toprule
    \multirow{2}{*}{Reference} & \multirow{2}{*}{Clipping threshold} & \multirow{2}{*}{$\lambda_k$} & Smoothness case: & Convergence rate & \multirow{2}{*}{LCR?} \\
    & & & $L_0 \,\,\, (?)\,\,\, cL_1  $ & $f(x^N) - f^* \lesssim$ \\ \midrule
    \multirow{5}{*}{\cite{Koloskova_2023}} & \multirow{5}{*}{arbitrary} & \multirow{2}{*}{$1$} & larger & $\mathcal{O}\left( \frac{L_0 R^2}{N} \right)$ & \redx \\
    & & & less or equal & $\mathcal{O}\left( \frac{\textcolor{red}{c} L_1 R^2}{N} \right)$ & \redx \\
    & & \multirow{2}{*}{$\frac{c}{\norms{\nabla f(x^k)}}$} & larger & $\mathcal{O}\left( \frac{L_0^2 \textcolor{red}{L} R^4}{c^2 N^2} \right)$ & \redx \\
    & & & less or equal & $\mathcal{O}\left( \frac{L_1^2 \textcolor{red}{L} R^4}{N^2} \right)$ & \redx \\ \midrule
    \cite{Gorbunov_2024} \& & \multirow{2}{*}{ $c = \frac{L_0}{L1}$ } & $1$ & equal & $\mathcal{O}\left( \frac{L_0 R^2}{N} \right)$ & \redx \\
    \cite{Vankov_2024} & & $\frac{c}{\norms{\nabla f(x^k)}}$ & equal & \redx & \redx \\ \midrule
    \multirow{6}{*}{\textbf{Theorem~\ref{th:Clip-GD} (Our work)}} & \multirow{6}{*}{arbitrary} & \multirow{2}{*}{$1$} & larger & $\mathcal{O}\left( \frac{L_0 R^2}{N} \right)$ & \redx \\
    & & &  less or equal & $\mathcal{O}\left( \min\left\{\frac{L_0 R^2}{N-T}, \left(1 - \frac{1}{L_1 R} \right)^{T} F_0\right\} \right)$ & \greencheck \\
    & & \multirow{2}{*}{$\frac{c}{\norms{\nabla f(x^k)}}$} & larger & $\mathcal{O}\left( \left(1 - \frac{c}{L_0 R} \right)^N F_0 \right)$ & \greencheck \\
    & & & less or equal & $\mathcal{O}\left( \left(1 - \frac{1}{L_1 R} \right)^N F_0 \right)$ & \greencheck \\
    \bottomrule
    \end{tabular}}
    \end{minipage}
\end{table*}

\vspace{-1em}
    \subsection{Notations and main assumptions}
    Before discussing related work, we first introduce the notations and assumptions that are used in this paper.

    \paragraph{Notations.} We use $\dotprod{x}{y}:= \sum_{i=1}^{d} x_i y_i$ to denote standard inner product of $x,y \in \mathbb{R}^d$. We denote Euclidean norm in $\mathbb{R}^d$ as $\| x\| := \sqrt{ \sum_{i=1}^d x_i^2} = \sqrt{\dotprod{x}{y}}$. We use $\ee_i \in \mathbb{R}^d$ to denote the $i$-th unit vector. For $\mL = (L^{(1)},\ldots, L^{(d)})^\top \in \R^d$ and $\alpha \in \R$, we define the norms ${\| x \|_{[\mL, \alpha]} := \sqrt{ \sum_{i=1}^d (L^{(i)})^\alpha x_i^2}}$ and ${\| x \|_{[L_p, \alpha]}^* := \sqrt{ \sum_{i=1}^d \frac{1}{(L_p^{(i)})^\alpha} x_i^2}}$. We denote by $\nabla f(x)$ the full gradient of function $f$ at point $x \in \mathbb{R}^d$, and by $\nabla_{i} f(x)$ the $i$-th coordinate gradient of function $f$ at point $x \in \mathbb{R}^d$. We also introduce $S^{\mL}_\alpha := \sum_{i}^d (L^{(i)})^\alpha$. We use $\tilde{O} (\cdot)$ to hide the logarithmic coefficients. We denote $f^* := f(x^*)$ and $x^* \in X^* \coloneqq \arg\min_{x\in \R^d} f(x)$ to be any solution of \eqref{eq:init_problem}. We also use $R \coloneqq \|x^0 - x^*\|$ and $F_0 \coloneqq f(x^0) - f^*$.

    The most common assumption about smoothness in the literature \citep[see, e.g.,][]{Nesterov_2013} is $L$-smoothness.
    \begin{assumption}[$L$-smoothness]\label{ass:L_smooth}
        Function $f$ is $L$-smooth if the following inequality is satisfied for any $x,y \in \mathbb{R}^d$:
        \begin{equation*}
            \norms{\nabla f(y) -\nabla f(x)} \leq L \norms{y - x}.
        \end{equation*}
    \end{assumption}

    However, instead of standard $L$-smoothness, we focus on the so-called $(L_0,L_1)$-smoothness \citep{Zhang_2019, Zhang_2020}.
    \begin{assumption}[$(L_0, L_1)$-smoothness]\label{ass:L0_L1_smooth}
        Function ${f: \mathbb{R}^d \rightarrow \mathbb{R}}$ is $(L_0, L_1)$-smooth if the following inequality is satisfied for any $x,y \in \mathbb{R}^d$ with $\norms{y - x} \leq \frac{1}{L_1}$:
        \begin{equation}
            \norms{\nabla f(y) -\nabla f(x)} \leq \left(L_0 + L_1 \norms{\nabla f(x)} \right) \norms{y - x}. \label{eq:L0_L1_smoothness}
        \end{equation}
    \end{assumption}
    If $L_1 = 0$, the above assumption recovers Assumption~\ref{ass:L_smooth} with $L = L_0$. Moreover, $(L_0,L_1)$-smoothness is strictly more general than $L$-smoothness, see the examples in \cite{Zhang_2019, Chen_2023, Koloskova_2023, Gorbunov_2024}. 

    Next, we also use a coordinate-wise version of Assumption~\ref{ass:L0_L1_smooth} introduced by \citet{crawshaw2022robustness}.
    \begin{assumption}[$(L_0, L_1)$-coordinate-smoothness]\label{ass:L0_L1_coordinate_smooth}
        A function ${f: \mathbb{R}^d \rightarrow \mathbb{R}}$ is $(L_0, L_1)$-coordinate-smooth for $L_0^{(1)}, L_0^{(2)}, ..., L_0^{(d)}, L_1^{(1)}, L_1^{(2)}, ..., L_1^{(d)} \geq 0$) if for any ${i \in [d]}$, $x \in \mathbb{R}^d$ and $h \in \mathbb{R}$, $|h| \leq \frac{1}{\max_{i\in[d]} L_1^{(i)}}$ the following inequality holds:
        \begin{equation*}
            |\nabla_i f(x + h \ee_i) -\nabla_i f(x)| \leq \left(L_0^{(i)} + L_1^{(i)} |\nabla_i f(x)| \right) |h|.
        \end{equation*}
    \end{assumption}
    The above assumption generalizes the standard coordinate $L$-smoothness \citep{Lin_2014, Allen_2016, Zhang_2017} similarly to how $(L_0,L_1)$-smoothness generalizes $L$-smoothness.
    
    We also assume that the function $f$~is~($\mu$-strongly)~convex.
    \begin{assumption}\label{ass:strongly_convex}
        Function $f: \mathbb{R}^d \rightarrow \mathbb{R}$ is $\mu \geq 0$ strongly convex if for any ${x,y \in \mathbb{R}^d}$ the following inequality holds:
        \begin{equation}
            f(y) \geq f(x) + \dotprod{\nabla f(x)}{y - x} + \frac{\mu}{2} \norms{y - x}^2.\label{eq:str_cvx}
        \end{equation}
    \end{assumption}
    Assumption~\ref{ass:strongly_convex} is classical and widely used in the literature \citep[see, e.g.,][]{Boyd_2004,Nesterov_2018}.
    
    \subsection{Paper structure}
    Further, our paper has the following structure. In Section~\ref{sec:Related Works}, we discuss the related work. In Section~\ref{sec:Full-Gradient Methods}, we provide the results for full-gradient methods. The case where the oracle only has access to the gradient coordinate or the comparison of the values of two functions is considered in Section~\ref{sec:Coordinate Descent Type Methods}. In Section~\ref{sec:Extension to Strongly Convex Setup}, we generalize our results for $(L_0,L_1)$-GD to the strongly convex case. Discussions of this work and future work plans are given in Section~\ref{sec:Discussion and Future Work}. 
    Section~\ref{sec:Conclusion} concludes the paper. All missing proofs of the theoretical results are provided in the Appendix. 
\section{Related Works}\label{sec:Related Works}
The literature on the analysis of GD-type methods is very rich. Below, we discuss only closely related works.
\paragraph{Full-gradient methods for the $(L_0,L_1)$-smooth convex optimization.} Although most of the existing works on $(L_0,L_1)$-smoothness focus on the non-convex case, there are several papers considering the (strongly) convex problems as well. \citet{Koloskova_2023} gives the first analysis Clip-GD \citep{Pascanu_2013} under $(L_0,L_1)$-smoothness and $L$-smoothness and proves $\cO\left(\max\left\{\frac{(L_0 + cL_1)R^2}{N}, \frac{R^4L(L_0 + cL_1)^2}{c^2N^2}\right\}\right)$ rate (see Table~\ref{tab:table_compare} for the details). This bound is derived under the additional $L$-smoothness assumption, which is not always satisfied for $(L_0,L_1)$-smooth problems. Moreover, when $\lambda_k = 1$ and $L_0 \leq cL_1$, the derived rate is proportional to $c$. In addition, the analysis from \citep{Koloskova_2023} implies the sublinear convergence rate for NGD (see the case $\lambda_k = \frac{c}{\norms{\nabla f(x^k)}}$ in Table~\ref{tab:table_compare}). \citet{Takezawa_2024} derive similar results for GD with Polyak Stepsizes (GD-PS), i.e., they show $\cO\left(\max\left\{\frac{L_0R^2}{N}, \frac{R^4LL_1^2}{c^2N^2}\right\}\right)$ convergence rate. Next, \citet{li2024convex} derive convergence rates for GD and its accelerated version under $(r,\ell)$-smoothness assumption, which generalizes $(L_0,L_1)$-smoothness. In particular, for GD \citet{li2024convex} prove $\cO(\frac{\ell R^2}{N})$ convergence rate, where $\ell = \cO(L_0 + L_1 G)$ and constant $G$ depends in $L_0,L_1, R, \|\nabla f(x^0)\|$, and $f(x^0) - f^*$, meaning that it can be exponentially large in terms of $L_1$ and $R$. Finally, \citet{Gorbunov_2024,Vankov_2024} independently improve the convergence rates of $(L_0,L_1)$-GD/Clip-GD by considering the special case of clipping radius $c=\frac{L_0}{L_1}$. More precisely, they prove $\cO\left(\frac{L_0 R^2}{N}\right)$ convergence rate if $N \geq L_1^2 R^2$ and extend this result to GD-PS (\citet{Vankov_2024} also show a similar result for NGD). However, the results from \cite{Gorbunov_2024,Vankov_2024} do not provide convergence rates in terms of $f(x^N) - f^*$ for the stage when $\|\nabla f(x^k)\| > \nicefrac{L_0}{L_1}$, which can be noticeable when $L_0$ is small and $L_1$ is large. \textit{In our work, we propose the analysis that addresses this limitation~(see~Table~\ref{tab:table_compare})}.
\paragraph{Coordinate descent type methods.} Convergence of coordinate methods is also relatively well-studied. For example, under the standard $L$-smoothness assumption ($(\nabla^2 f(x))_{i,i} \leq L$), the coordinate descent (CD) method has the following convergence rate $\mathcal{O} \left( \frac{d L R^2}{N} \right)$ \citep[see, e.g.,][]{Bubeck_2015}. Using the fact that $\frac{1}{d}\sum_{i=1}^{d} L^{(i)} \leq L$ and assuming $L$-coordinate-smoothness ($(\nabla^2 f(x))_{i,i} \leq L^{(i)}$), the previous result can be improved to $\mathcal{O} \left( \frac{\sum_{i=1}^d L^{(i)} R^2}{N} \right)$ rate. Next, assuming that the active coordinate $i_k$ can be obtained (independently) from the distribution $p_{\alpha}(i) = \nicefrac{(L^{(i)})^\alpha}{S_{\alpha}}$, then RCD converges at $\mathcal{O}\left( \frac{S_{\alpha} R_{[\mL, 1-\alpha]}^2}{N} \right)$ rate \cite{Nesterov_2012}, where\newline $R_{[\mL, 1-\alpha]} \coloneqq \max_{x\in \R^d}\left\{ \max_{x^* \in X^*} \|x - x^*\|_{[\mL,1-\alpha]}~:~f(x) \leq f(x^0) \right\}$. Moreover, \citet{Lobanov_2024} show that it is possible to create an OrderRCD algorithm based on RCD, whose oracle has access only to function comparisons (this oracle can be motivated by, e.g., RLHF \citep{ouyang2022training, bai2022training}). More precisely, \citet{Lobanov_2024} prove that the iteration complexity of OrderRCD is the same as for RCD, and the oracle complexity is inferior only in $\log(\nicefrac{1}{\epsilon})$ factor, where $\epsilon$ is the accuracy of the solution of the linear search problem. \textit{In our paper, we extend these results to the more general case of $(L_0,L_1)$-smoothness}. 

\section{Full-Gradient Methods}\label{sec:Full-Gradient Methods}
In this section, we present our result for full-gradient algorithms (see GD in Subsection~\ref{subsec:Gradient descent Method}, NGD in Subsection~\ref{subsec:Normalized GD method}, and Clip-GD in Subsection~\ref{subsec:Clipped GD method}).

\subsection{Gradient descent method}\label{subsec:Gradient descent Method}
    The first algorithm we consider has the following algorithm:
    \begin{algorithm}[H]
           \caption{Gradient Descent Method (GD)}
           \label{algo:GD}
        \begin{algorithmic}
           \State {\bfseries Input:} $x_0 \in \mathbb{R}^d$, iterations number $N$, step size $\eta_k >0$
           
           \For{$k=0$ {\bfseries to} $N-1$}
           \State $x^{k+1} \gets x^{k} - \eta_{k} \nabla f(x^k)$
           \EndFor
           \State {\bfseries Return:} $x^{N}$
        \end{algorithmic}
    \end{algorithm}  
    We prove the following result for Algorithm~\ref{algo:GD} with stepsize $\eta_k = (L_0+L_1 \norms{\nabla f(x^k)})^{-1}$ (to emphasize the specificity of the step size we call it $(L_0,L_1)$-GD for brevity).

    \begin{theorem}\label{th:GD_convex}
        Let function $f$ satisfy Assumption~\ref{ass:L0_L1_smooth} ($(L_0,L_1)$-smoothness) and Assumption~\ref{ass:strongly_convex} (convexity, ${\mu=0}$), then GD (Algorithm~\ref{algo:GD})  with step size ${\eta_k = (L_0 + L_1 \norms{\nabla f(x^k)})^{-1}}$ guarantees
        \begin{itemize}
            \item linear convergence, if $\|\nabla f(x^k)\| \geq \frac{L_0}{L_1}$ for $k \in [N-1]$
        \begin{boxF}
            \begin{equation*}
                f(x^N) - f^* \leq \left( 1 - \frac{1}{4 L_1 R} \right)^N F_0;
            \end{equation*}
        \end{boxF}

            \item sublinear convergence, if $\|\nabla f(x^{N-1})\| < \frac{L_0}{L_1}$:
            \begin{boxF}
            \begin{equation*}
                f(x^N) - f^* < \frac{4 L_0 R^2}{N}.
            \end{equation*} 
            \end{boxF}
        \end{itemize}

        In the general case, the convergence rate is
        \begin{boxA}
            \begin{equation*}
                f(x^N) - f^* \leq \min\left\{\frac{4 L_0 R^2}{N-T}, \left( 1 - \frac{1}{4 L_1 R} \right)^T F_0\right\},
            \end{equation*}
        \end{boxA}
        where $T\geq 0$ is the smallest index such as $\|\nabla f(x^T)\| < \frac{L_0}{L_1}$.
    \end{theorem}
    Given the monotonicity of the gradient norm (see Appendix~\ref{app:Monotonicity of Algorithms by Gradient Norms}), Theorem~\ref{th:GD_convex} characterizes in details the convergence behavior of GD for convex $(L_0,L_1)$-smooth problems. More precisely, as long as the gradient norm is larger than $\nicefrac{L_0}{L_1}$, GD converges with linear rate, but when the method approaches the solution ($\|\nabla f(x^k)\| < \nicefrac{L_0}{L_1}$) the convergence slows down to the standard sublinear rate. That is, $\mathcal{O}(\frac{L_0 R^2}{N})$ rate is common to the previous works \citep{Gorbunov_2024,Vankov_2024} (see Table~\ref{tab:table_compare}). However, in contrast to \citep{Gorbunov_2024,Vankov_2024}, our analysis shows $\cO (1 - \nicefrac{1}{L_1 R})^N F_0$ rate when $\|\nabla f(x^k)\| \geq \nicefrac{L_0}{L_1}$ (see the case of $\lambda_k = \nicefrac{c}{\norms{\nabla f(x^k)}}$ in Table~\ref{tab:table_compare}). Moreover, our result significantly improves the one from \citep{Koloskova_2023} (see smoothness case ``less or equal" with $\lambda_k = \nicefrac{c}{\norms{\nabla f(x^k)}}$ in Table~\ref{tab:table_compare}) from sublinear to linear and gets rid of potentially large parameter $c \geq \nicefrac{L_0}{L_1}$.  The proof of the Theorem~\ref{th:GD_convex} is provided in Appendix~\ref{subsec:Proof_GD}.  

    The significance of the improved estimate can be observed under the assumption of strong growth condition\footnote{We refer to Assumption~\ref{ass:L0_L1_smooth} with $L_0 = 0$ as strong growth condition for smoothness assumption by analogy with \cite{Vaswani_2019} for variance.}.
    \begin{boxD}
    \begin{remark}\label{rem:GD}
        Theorem~\ref{th:GD_convex} implies that under Assumption~\ref{ass:L0_L1_smooth} with $L_0 = 0$, Algorithm~\ref{algo:GD} converges to the desired accuracy $\varepsilon$ (${f(x^N) - f^* \leq \varepsilon}$) after $N = \mathcal{O}\left( L_1 R \log \frac{F_0}{\varepsilon} \right)$ iterations.
    \end{remark}
    \end{boxD}
    The result of Remark~\ref{rem:GD} significantly outperforms all known results in this regime. In particular, \citet{Koloskova_2023} show $\mathcal{O} (\nicefrac{L_1 c R^2}{\varepsilon})$ complexity bound for Clip-GD, and \citet{Gorbunov_2024,Vankov_2024} do not provide explicit rates in this case.

    \subsection{Normalized GD method}\label{subsec:Normalized GD method}
    From the previous section, we see that GD with step size with step size $\eta_k = (L_0 + L_1 \norms{\nabla f(x^k)})^{-1}$ enjoys linear convergence in the convex setting, when $\|\nabla f(x^k)\| \geq \nicefrac{L_0}{L_1}$. However, in this regime, we have $L_1 \norms{\nabla f(x^k)} \geq L_0$, meaning that $(2L_1 \norms{\nabla f(x^k)})^{-1} \leq \eta_k \leq (L_1 \norms{\nabla f(x^k)})^{-1}$, i.e., the method is very close to NGD (Algorithm~\ref{algo:NGD}). Therefore, it is natural to expect similar behavior from NGD as for GD. 
    \begin{algorithm}[H]
           \caption{Normalized Gradient Descent Method (NGD)}
           \label{algo:NGD}
        \begin{algorithmic}
           \State {\bfseries Input:} $x_0 \in \mathbb{R}^d$, iterations number $N$, step size $\eta_k >0$
           
           \For{$k=0$ {\bfseries to} $N-1$}
           \If{$\|\nabla f(x^k) \| = 0$}
           \State {\bfseries Return:} $x^{k}$
           \EndIf
           \State $x^{k+1} \gets x^{k} - \eta_{k} \frac{\nabla f(x^k)}{\norms{\nabla f(x^k)}}$
           \EndFor
           \State {\bfseries Return:} $x^{N}$
        \end{algorithmic}
    \end{algorithm}
    The following result formalizes this observation.
    \begin{theorem}\label{th:NGD}
        Let function $f$ satisfy Assumption~\ref{ass:L0_L1_smooth} ($(L_0,L_1)$-smoothness) and Assumption~\ref{ass:strongly_convex} (convexity, ${\mu=0}$), then Algorithm~\ref{algo:NGD} with step size $\eta_k = \eta \leq \nicefrac{c}{(L_0 + L_1 c)}$, where constant $c>0$ is such that $\|\nabla f(x^k) \| \geq c$ for all $k = 0,1,\ldots,N-1$, has~linear~convergence:
        \begin{boxF}
            \begin{equation*}
                f(x^N) - f^* \leq \left( 1 - \frac{\eta}{2 R} \right)^N F_0.
            \end{equation*}
        \end{boxF}
    \end{theorem}
    Theorem~\ref{th:NGD} shows that in the case of $c \geq \nicefrac{L_0}{L_1}$, NGD has $\mathcal{O}\left((1 - \frac{1}{L_1 R})^N F_0\right)$ convergence rate similarly to GD, which is natural to expect due to $\|\nabla f(x^k)\| \geq c$ and the discussion given in the beginning of this subsection. However, if we select large enough $N$, one has to select $c$ small enough such that $\|\nabla f(x^k) \| \geq c$ holds for all $k = 0,1,\ldots,N-1$. If $c < \nicefrac{L_0}{L_1}$, then the rate reduces to $\mathcal{O}\left((1 - \nicefrac{c}{(L_0 R)})^N F_0\right)$ and the method is guaranteed to converge only to the error $\varepsilon \sim cR$. Therefore, to guarantee the convergence to $\varepsilon$-accuracy, one has to take $c \sim \nicefrac{\varepsilon}{R}$ in the worst case. In this case, our result implies $\mathcal{O} (\nicefrac{L_0 R^2\log(\nicefrac{F_0}{\varepsilon})}{\varepsilon})$ complexity for NGD. However, hyperparameter $c$ depends only on the gradient norm, so in problems where the high accuracy on gradient norm is not required, Algorithm~\ref{algo:NGD} is efficient~and~shows linear convergence. The proof of Theorem~\ref{th:NGD} see~Appendix~\ref{subsec:Proof_NGD}.
    \begin{boxD}
    \begin{remark}\label{rem:NGD}
        Theorem~\ref{th:NGD} implies that under Assumption~\ref{ass:L0_L1_smooth} with $L_0 = 0$, Algorithm~\ref{algo:NGD} converges to the desired accuracy $\varepsilon$ (${f(x^N) - f^* \leq \varepsilon}$) after $N = \mathcal{O}\left( L_1 R \log \frac{F_0}{\varepsilon} \right)$ iterations.
    \end{remark}
    \end{boxD}
    As previously noted, when $\norms{\nabla f(x^k)} \geq \frac{L_0}{L_1}$ GD and NGD with $c \geq \frac{L_0}{L_1}$ are almost the same. Therefore, the result of Remark~\ref{rem:NGD} is expected given Remark~\ref{rem:GD}. Moreover, our results imply that NGD has $\mathcal{O} (\max\{\nicefrac{L_0 R^2\log(\nicefrac{F_0}{\varepsilon})}{\varepsilon}, L_1R\log(\nicefrac{F_0}{\varepsilon})\})$ complexity. Compared to $\mathcal{O} (\max\{\nicefrac{L_0 \overline{R}^2}{\varepsilon}, L_1^2\overline{R}^2\})$ complexity bound derived for NGD with $\eta_k = \nicefrac{\hat{R}}{\sqrt{N+1}}$, $\overline{R} \coloneqq \hat{R} + \nicefrac{R^2}{\hat{R}}$ by \citet{Vankov_2024}, our bound has an additional logarithmic factor in the first term but has much better second term when $L_1 R$ is large and $\log(\nicefrac{F_0}{\varepsilon})$ is much smaller than $L_1 R$.

    \subsection{Clipped GD method}\label{subsec:Clipped GD method}
    In this section, we consider Clip-GD (Algorithm~\ref{algo:Clip-GD}), which applies the clipping operator to the gradient:
    \begin{equation}\label{eq:clip}
        \text{clip}_c(\nabla f(x)) = \min \left\{ 1, \frac{c}{\norms{\nabla f(x)}} \right\} \nabla f(x),
    \end{equation}
    where $c>0$ is the clipping radius. Clip-GD can also be seen as a combination of GD (when $\|\nabla f(x^k)\| \leq c$) and NGD (when $\|\nabla f(x^k)\| > c$).
    \begin{algorithm}[H]
           \caption{Clipped Gradient Descent Method (Clip-GD)}
           \label{algo:Clip-GD}
        \begin{algorithmic}
           \State {\bfseries Input:} initial point $x_0 \in \mathbb{R}^d$, iterations number $N$, step size $\eta_k >0$ and clipping radius $c>0$
           
           \For{$k=0$ {\bfseries to} $N-1$}
           \State $x^{k+1} \gets x^{k} - \eta_{k} \cdot \text{clip}_c (\nabla f(x^k))$ according to \eqref{eq:clip}
           \EndFor
           \State {\bfseries Return:} $x^{N}$
        \end{algorithmic}
    \end{algorithm}  

    Then, following similar reasoning as in the previous sections, we obtain the next convergence result for Clip-GD method.
    \begin{theorem}\label{th:Clip-GD}
        Let function $f$ satisfy Assumption~\ref{ass:L0_L1_smooth} ($(L_0,L_1)$-smoothness) and Assumption~\ref{ass:strongly_convex} (convexity, ${\mu=0}$), then Algorithm~\ref{algo:Clip-GD}  with step size $\eta_k = (L_0 + L_1 \min\{ \| \nabla f(x^k) \|,c\})^{-1}$ guarantees the following error:
        \begin{boxF}
        \begin{equation*}
           f(x^N) - f^* = \mathcal{O} \left(\min\left\{\frac{L_0R^2}{N-T}, \left( 1 - \frac{\rho}{R} \right)^{T} F_0 \right\}\right),
        \end{equation*}
        \end{boxF}
        where $\rho \coloneqq \nicefrac{c}{\max\{L_0, L_1 c \}}$ and $T\geq 0$ is the smallest index such as $\|\nabla f(x^T)\| < \min\{c,\nicefrac{L_0}{L_1}\}$
    \end{theorem}
    Since NGD and GD are monotonically decreasing in terms of the gradient norm, it follows that Algorithm~\ref{algo:Clip-GD} is also monotonically decreasing in terms of the gradient norm (see Appendix~\ref{app:Monotonicity of Algorithms by Gradient Norms} for details). Given this fact, Theorem~\ref{th:Clip-GD} shows that Algorithm~\ref{algo:Clip-GD} has two convergence regimes depending on the ratio of $c$ and $\nicefrac{L_0}{L_1}$. If $c \geq \nicefrac{L_0}{L_1}$, then Clip-GD starts its convergence with a linear rate $\mathcal{O} \left(\left( 1 - (\nicefrac{1}{L_1 R}) \right)^N F_0 \right)$, and as soon as it approaches the solution, i.e., when $\norms{\nabla f(x^k)} < \nicefrac{L_0}{L_1}$, it slows down to a sublinear $\mathcal{O} \left( \nicefrac{L_0 R^2}{N} \right)$ rate. If $c < \nicefrac{L_0}{L_1}$, then Clip-GD has inferior linear convergence rate $\mathcal{O} \left(\left( 1 - (\nicefrac{c}{L_0 R}) \right)^N F_0 \right)$ at the beginning, and approaching the solution, i.e., when $\norms{\nabla f(x^k)} < c$, it slows down to the same sublinear rate. The cases in Appendix~\ref{subsec:Proof_ClipGD} are discussed in more detail. Table~\ref{tab:table_compare} summarizes the derived results and compares them with the closely related works analyzing Clip-GD. It is worth noting that Theorem~\ref{th:Clip-GD} shows when Algorithm~\ref{algo:Clip-GD} has linear convergence and gets rid of standard smoothness constant $L$ (in contrast to \citep{Koloskova_2023}). Moreover, Theorem~\ref{th:Clip-GD} is valid for an arbitrary clipping threshold $c$ (in contrast to \citep{Gorbunov_2024,Vankov_2024}). 
    \begin{boxD}
    \begin{remark}[Strong growth condition]\label{rem:Clip-GD}
        Theorem~\ref{th:Clip-GD} implies that under Assumption~\ref{ass:L0_L1_smooth} with $L_0 = 0$, Algorithm~\ref{algo:Clip-GD} converges to the desired accuracy $\varepsilon$ (${f(x^N) - f^* \leq \varepsilon}$) after $N = \mathcal{O}\left( L_1 R \log \frac{F_0}{\varepsilon} \right)$.
    \end{remark}
    \end{boxD}

\section{Coordinate Descent Type Methods}\label{sec:Coordinate Descent Type Methods}
In this section, we present our main results for the algorithms that does not use access to the full gradient (see RCD in Subsection~\ref{subsec:Random coordinate descent}, and OrderRCD see Subsection~\ref{subsec:Random coordinate descent with Order Oracle}).

    \subsection{Random coordinate descent}\label{subsec:Random coordinate descent}
    RCD is formalized as Algorithm~\ref{algo:RCD}. At each iteration, the method computes the gradient coordinate $\nabla_{i_k} f(x^k)$, where active coordinate $i_k$ is selected uniformly at random from $[d]$ independently from previous steps.
    \begin{algorithm}[H]
           \caption{Random Coordinate Descent Method (RCD)}
           \label{algo:RCD}
        \begin{algorithmic}
           \State {\bfseries Input:} initial point $x_0 \in \mathbb{R}^d$, iterations number $N$, step size $\eta_k > 0$
           
           \For{$k=0$ {\bfseries to} $N-1$}
           \State{1.}~~~sample $i_k$ uniformly at random from $[d]$
           \State{2.}~~~$x^{k+1} \gets x^{k} - \eta_{k} \nabla_{i_k} f(x^k) \ee_{i_k}$
           \EndFor
           \State {\bfseries Return:} $x^{N}$
        \end{algorithmic}
    \end{algorithm}  
    Our main results for RCD are given below.
    \begin{theorem}\label{th:RCD}
         Let function $f$ satisfy Assumption~\ref{ass:L0_L1_coordinate_smooth} ($(L_0,L_1)$-coordinate-smoothness) and Assumption~\ref{ass:strongly_convex} (convexity, ${\mu=0}$), then RCD (Algorithm~\ref{algo:RCD})  with step size ${\eta_k \leq (L_0 + L_1 |\nabla_{i_k} f(x^k)|)^{-1}}$, where $L_0 = \max_{i\in [d]}L_0^{(i)}$ and $L_1 = \max_{i\in [d]}L_1^{(i)}$, guarantees the following error:
         \begin{boxF}
                 \begin{equation*}
                     \hspace{-1.2em}\mathbb{E}\left[ f(x^{N}) \right] - f^* = \cO\!\left(\!\max\!\left\{\!\left(1 - \frac{\rho}{dR}\right)^{N} \! F_0, \frac{dL_0 R^2}{N}\!\right\}\!\right),
                 \end{equation*}
             \end{boxF}
             where $\rho \coloneqq \nicefrac{1}{(4\sqrt{2}L_1)}$.
    \end{theorem}
    Theorem~\ref{th:RCD} provides a generalization of the results of \cite{Nesterov_2012} to the case of $(L_0,L_1)$-coordinate-smoothness (see Assumption~\ref{ass:L0_L1_coordinate_smooth}). In particular, following Section~\ref{sec:Full-Gradient Methods}, we separated $L_0$ and $L_1$ in the convergence results and also show that there is no need to assume standard $L$-smoothness since the case $L_1 = 0$ covers it. Moreover, in the case of $L_0$ being much smaller than $L_1$, the results of Theorem~\ref{th:RCD} are strictly better than previously known ones. Furthermore, in the case of $L_0 = 0$, RCD converges linearly to any accuracy.  
    \begin{boxD}
    \begin{remark}[Strong growth condition]\label{rem:RCD}
        Theorem~\ref{th:RCD} implies that under Assumption~\ref{ass:L0_L1_coordinate_smooth} with $L_0 = 0$, Algorithm~\ref{algo:RCD} converges to the desired accuracy $\varepsilon$ (${\mathbb{E}[f(x^N)] - f^* \leq \varepsilon}$) after ${N = \mathcal{O}\left( dL_1R \log \frac{F_0}{\varepsilon} \right)}$ iterations.
    \end{remark}
    \end{boxD}
    For a detailed proof of Theorem~\ref{th:RCD}, see Appendix~\ref{subsec:proof_RCD}.

    \subsection{Random coordinate descent with Order Oracle}\label{subsec:Random coordinate descent with Order Oracle}
    In this section, we consider the OrderRCD (Algorithm~\ref{algo:OrderRCD}). In contrast to all previously considered methods in this paper, OrderRCD does not have access to a first-order oracle. Instead, the algorithm uses so-called Order Oracle: for any $x,y\in \R^d$, one can compute
    \begin{equation}\label{eq:Order_Oracle}
        \psi (x,y) = \text{sign} \left[f(y) - f(x) \right]
    \end{equation}
    \begin{algorithm}[H]
           \caption{RCD with Order Oracle (OrderRCD)}
           \label{algo:OrderRCD}
        \begin{algorithmic}
           \State {\bfseries Input:} initial point $x_0 \in \mathbb{R}^d$, iterations number $N$, random generator $\mathcal{R}_{\alpha}(L_0,L_1)$
           
           \For{$k=0$ {\bfseries to} $N-1$}
           \State{1.}~~~sample $i_k$ uniformly at random from $[d]$
           \State{2.}~~~compute $\zeta_k = \text{argmin}_{\zeta} \{ f(x^k + \zeta \ee_{i_k})\}$~via~(GRM)
           \State{3.}~~~$x^{k+1} \gets x^{k} + \zeta_{k} \ee_{i_k}$
           \EndFor
           \State {\bfseries Return:} $x^{N}$
        \end{algorithmic}
    \end{algorithm}  

    Algorithm~\ref{algo:OrderRCD} is similar to Algorithm~\ref{algo:RCD}, but it does not have access to the gradient coordinate $\nabla_{i_k} f(x^k)$. Following \citet{Lobanov_2024}, we address this challenge using the standard steepest descent trick, namely, we solve at each iteration the auxiliary linear search problem using the golden ratio method (GRM, see Algorithm~\ref{algo:GRM} in Appendix~\ref{subsec:Proof_OrderRCD}) with $\epsilon$ accuracy allowing to match RCD with step~size~$\eta_k$.

    Below, we present the convergence result for Algorithm~\ref{algo:OrderRCD}.
    \begin{theorem}\label{th:OrderRCD}
         Let function $f$ satisfy Assumption~\ref{ass:L0_L1_coordinate_smooth} ($(L_0,L_1)$-coordinate-smoothness) and Assumption~\ref{ass:strongly_convex} (convexity, ${\mu=0}$), then Algorithm~\ref{algo:OrderRCD} (OrderRCD)  with oracle~\eqref{eq:Order_Oracle} guarantees the following error:
         \begin{boxF}
                 \begin{equation*}
                     \hspace{-1.2em}\mathbb{E}\left[ f(x^{N}) \right] - f^* = \cO\!\left(\!\max\!\left\{\!\left(1 - \frac{\rho}{dR}\right)^{N} \! F_0, \frac{dL_0 R^2}{N}\!\right\}\!\right),
                 \end{equation*}
             \end{boxF}
             where $\rho \coloneqq \nicefrac{1}{(4\sqrt{2}L_1)}$, $L_0 = \max_{i\in [d]}L_0^{(i)}$, and $L_1 = \max_{i\in [d]}L_1^{(i)}$.
    \end{theorem}
    That is, Theorem~\ref{th:OrderRCD} gives exactly the same rate as Theorem~\ref{th:RCD} with one exception. However, it is important to note that Algorithm~\ref{algo:OrderRCD} requires $\log(\nicefrac{1}{\epsilon})$ oracle calls per iteration to solve the linear search problem at each iteration using GRM, where Order Oracle~\eqref{eq:Order_Oracle} is directly used. In the special case of $L_1 = 0$, Theorem~\ref{th:OrderRCD} recovers known results, e.g., \cite{Gorbunov_2019,Saha_2021}. However, when $L_0$ is much smaller than $L_1$, Theorem~\ref{th:OrderRCD} shows better results, i.e., linear convergence.
    \begin{boxD}
    \begin{remark}[Strong growth condition]\label{rem:OrderRCD}
        Theorem~\ref{th:OrderRCD} implies that under Assumption~\ref{ass:L0_L1_coordinate_smooth} with $L_0 = 0$, Algorithm~\ref{algo:OrderRCD} converges to the desired accuracy $\varepsilon$ (${\mathbb{E}[f(x^N)] - f^* \leq \varepsilon}$) after ${N = \mathcal{O}\left( dL_1R \log \frac{F_0}{\varepsilon} \right)}$ iterations and ${T = \mathcal{O}\left( N \log\frac{1}{\epsilon} \right)}$ oracle calls.
    \end{remark}
    \end{boxD}
    For a detailed proof of Theorem~\ref{th:RCD}, see Appendix~\ref{subsec:Proof_OrderRCD}.

\section{Extension to Strongly Convex Setup}\label{sec:Extension to Strongly Convex Setup}
In this section, we answer the question:
\begin{center}
    \textit{“Are there convergence improvements of algorithms under the $(L_0,L_1)$-smoothness assumption compared to standard smoothness in a strongly convex setup?”}
\end{center} 
In particular, we consider GD (Algorithm~\ref{algo:GD}) and derive the following convergence result.
\begin{theorem}\label{th:strongly_convex}
    Let function $f$ satisfy Assumption~\ref{ass:L0_L1_smooth} ($(L_0,L_1)$-smoothness) and Assumption~\ref{ass:strongly_convex} (strongly convexity, ${\mu>0}$), then gradient descent method (Algorithm~\ref{algo:GD})  with step size $\eta_k = (L_0 + L_1 \norms{\nabla f(x^k)})^{-1}$ guarantees:
    \begin{boxA}
            \begin{equation*}
                \hspace{-0.4em}F_N \leq \left( 1 - \rho_3 \right)^{N - \cT_2}\left(1 - \rho_2\right)^{\cT_2 - \cT_1}\left( 1 - \rho_1 \right)^{\cT_1 + 1} F_0.
            \end{equation*}
    \end{boxA}
    where $F_k = f(x^k) - f^*$ for $k \in [N]$, $\rho_3 = \frac{\mu}{2L_0}$, $\rho_2 = \max\left\{\frac{\sqrt{\mu}}{2\sqrt{2}L_1}, \frac{1}{4L_1R}\right\}$, $\rho_1 = \frac{1}{4L_1R}$, $N_3 = N-\cT_2$, $\cT_2 = \max\{k \in [N-1]\mid \norms{\nabla f(x^k)} \geq \frac{L_0}{L_1}\}$ (if there are no such $k$, we let $\cT_2 = -1$), and $\cT_1 = \max\{k \in [N-1]\mid \norms{\nabla f(x^k)} \geq \frac{L_0}{L_1} \text{ and } F_k > 1\}$ (if there are no such $k$, we let $\cT_1 = -1$). In particular, 
    \begin{itemize}
            \item if $\cT_1 = -1$ and $\cT_2 = N-1$, then 
        \begin{boxF}
            \begin{equation*}
            \hspace{-0.5em}
                F_{N} \lesssim \left( 1 - \max\left\{\frac{\sqrt{\mu}}{2 \sqrt{2} L_1}, \frac{1}{4 L_1R}\right\} \right)^{N} F_{0};
            \end{equation*}
        \end{boxF}
    
        \item if $\cT_1 = \cT_2 = -1$, then
            \begin{boxF}
            \begin{equation*}
                F_{N} \lesssim \left( 1 - \frac{\mu}{2 L_0}\right)^{N} F_{0}.
            \end{equation*} 
            \end{boxF}
    \end{itemize}
\end{theorem}
The above theorem improves the result from \citep{Gorbunov_2024} that show $\|x^N - x^*\|^2 = \cO((1 - \rho_3)^{N-\cT_2}R)$ rate for GD, and, in contrast to the result from \citep{Koloskova_2023}, Theorem~\ref{th:strongly_convex} does not require $L$-smoothness. Moreover, the derived bound contains factor $(1 - \rho_2)^{\cT_2 - \cT_1}(1-\rho_1)^{\cT_1+1}$, which might be better than $(1 - \rho_1)^{\cT_2}$ when $R > \sqrt{\nicefrac{2}{\mu}}$. Moreover, if $\nicefrac{\mu}{2L_0} < \nicefrac{\sqrt{\mu}}{(2\sqrt{2}L_1)}$, then the derived result is strictly better than the known ones for GD under the standard smoothness. The proof of the Theorem~\ref{th:strongly_convex} is provided in Appendix~\ref{app:Proof_GD_strongly}.

\section{Discussion and Future Work}\label{sec:Discussion and Future Work}
In this paper (see Sections~\ref{sec:Full-Gradient Methods} and \ref{sec:Coordinate Descent Type Methods}) we have shown that linear convergence in a convex setup is possible in the case of $(L_0,L_1)$-smooth problems with small enough $L_0$. However, looking at the convergence of Algorithms~\ref{algo:GD}-\ref{algo:OrderRCD}, in particular Theorem~\ref{th:GD_convex}-\ref{th:OrderRCD}, we see that the dominant part is sublinear $\mathcal{O}\left( \nicefrac{1}{N} \right)$ and might be further improved. Nevertheless, as Remarks~\ref{rem:GD}-\ref{rem:OrderRCD} demonstrate, in the case of the strong growth condition ($L_0 = 0$), we can observe significant improvements compared to previous works by \citep[see e.g., ][]{Koloskova_2023, Gorbunov_2024, Vankov_2024}. A prime example of a function that satisfies the strong growth condition is logistic function \citep[see Example 1.6, ][]{Gorbunov_2024}, which is classical in the field of machine learning.

As future work, we see the following directions: generalizing Algorithms~\ref{algo:NGD}-\ref{algo:OrderRCD} to the strongly convex case; investigating whether the proposed technique can be used in the analysis of stochastic methods; investigating the convergence advantages of assuming generalized smoothness (Assumption~\ref{ass:L0_L1_smooth}) in accelerated optimization algorithms and many other directions. We believe this work opens up a number of research directions, including answering the question of whether it is possible to create adaptive methods,~as~well~as methods in different settings (such as federated learning, overparameterization, etc.)~that~will~exhibit~similar~advantages.

\section{Conclusion}\label{sec:Conclusion}
This paper demonstrates that generalized smoothness allows us to achieve linear convergence rates in convex setups. We explained the convergence behavior of gradient descent theoretically and showed that the advantages of generalized smoothness extend to gradient descent method variants, in particular, we significantly improved convergence estimates for GD, NGD, Clip-GD, and demonstrated novel convergence results for algorithms that do not have access to the full gradient as well as to the function values themselves (RCD and OrderRCD). We have demonstrated that this work opens up a number of directions for future research (see Section~\ref{sec:Discussion and Future Work}).

\section*{Acknowledgments}
The authors would like to thank Anastasia Koloskova and  Nazarii Tupitsa for useful discussions.

%% file: 4_appendix.tex
\newpage
\appendix

\section{Auxiliary Results}
In this section, we provide auxiliary technical results that are used in our analysis.

\paragraph{Basic inequalities.} For all $a,b \in \mathbb{R}^d$ ($d \geq 1$), the following inequalities hold:
\begin{equation}
    \label{eq:qudrat_raznosti}
    2 \dotprod{a}{b} - \| b \|^2 = \| a \|^2 - \|a - b \|^2,
\end{equation}
\begin{equation}
    \label{eq:scalar_product_bound}
    \dotprod{a}{b} \leq \| a \| \cdot \| b \|.
\end{equation}

\paragraph{Generalized-Lipschitz-smoothness.} In the analysis of full-gradient methods, we assume that the $(L_0,L_1)$-smoothness condition (Assumption \ref{ass:L0_L1_smooth}) is satisfied. This inequality can be represented in the equivalent form for any $x,y \in \mathbb{R}^d$ \citep{Zhang_2020}:
\begin{equation}
    \label{eq:ass_smooth}
    f(y) - f(x) \leq \dotprod{\nabla f(x)}{ y - x} + \frac{L_0 + L_1 \norms{\nabla f(x)}}{2} \norms{y-x}^2,
\end{equation}
where $L_0, L_1 \geq 0$ for any $x \in \mathbb{R}^d$ and $\norms{y-x} \leq \frac{1}{L_1}$.

\paragraph{Generalized-coordinate-Lipschitz-smoothness.} In the analysis of coordinate-wise methods, we assume that the smoothness condition (Assumption \ref{ass:L0_L1_coordinate_smooth}) is satisfied. This inequality can be represented in the equivalent form \citep[Lemma 1]{crawshaw2022robustness}:
\begin{equation}
    \label{eq:ass_smooth_coordinate}
    f(x+ h \ee_i) \leq f(x) + h \nabla_{i} f(x) + \frac{\left( L^{(i)}_0 + L_1^{(i)} |\nabla_i f(x)|\right) h^2}{2},
\end{equation}
where $L_0^{(1)}, L_0^{(2)}, ..., L_0^{(d)}, L_1^{(1)}, L_1^{(2)}, ..., L_1^{(d)} \geq 0$ for any $i \in [d], x \in \mathbb{R}^d$ and $|h| \leq \frac{1}{L_1^{(i)}}$.

\section{Monotonicity of Gradient Norms}\label{app:Monotonicity of Algorithms by Gradient Norms}

In this section, we give a proof of monotonicity of convergence of algorithms by gradient norm. In particular, see Lemma~\ref{lem:GD} for the proof for Algorithm~\ref{algo:GD}, see Lemma~\ref{lem:NGD} for Algorithm~\ref{algo:NGD}, and see Lemma~\ref{lem:ClipGD} for Algorithm~\ref{algo:Clip-GD}.

First of all, we start with the auxiliary result.
\begin{lemma}\label{lem:auxiliary_result}
    Let function $f$ satisfy Assumption~\ref{ass:L0_L1_smooth} ($(L_0,L_1)$-smoothness) and Assumption~\ref{ass:strongly_convex} (convexity, ${\mu=0}$), then for $x,y \in \mathbb{R}^d$ such that $\|y - x\| \leq \frac{1}{L_1}$ we have:
    \begin{equation}
        \frac{\norms{\nabla f(y) - \nabla f(x)}^2}{2 \left(L_0 + L_1 \norms{\nabla f(x)}\right)} \leq f(x) - f(y) - \dotprod{\nabla f(y)}{x - y}. \label{eq:L0_L1_cocoercivity}
    \end{equation}
\end{lemma}
\begin{proof}
    The proof of this statement is based on the results from \cite{Nesterov_2018, Gorbunov_2024}.

    Let us define the following function $\varphi_{a}(b)$ for a given $a \in \mathbb{R}^d$:
    \begin{equation*}
        \varphi_{a}(b) = f(b) - \dotprod{\nabla f(a)}{b}.
    \end{equation*}
    Then this function is differentiable and $\nabla \varphi_{a}(b) = \nabla f(b) - \nabla f(a)$. Moreover, for any $b, c \in \mathbb{R}^d$ such that $\|b - c\| \leq \frac{1}{L_1}$ we have:
    \begin{equation}\label{eq:varphi_smooth}
        \norms{\nabla \varphi_{a}(c) - \nabla \varphi_{a}(b)} = \norms{\nabla f(b) - \nabla f(c)} \overset{\circledOne}{\leq} \left(L_0 + L_1 \norms{\nabla f(c)}\right) \norms{b - c},
    \end{equation}
    where in $\circledOne$ we applied Assumption~\ref{ass:L0_L1_smooth}. Next, for given $a$ and for any $b,c \in \mathbb{R}^d$ such that $\|b - c\| \leq \frac{1}{L_1}$ we define function $\psi_{abc}(t): \mathbb{R} \rightarrow \mathbb{R}$ as
    \begin{equation*}
        \psi_{abc}(t) = \varphi_{a}(c + t(b - c)).
    \end{equation*}
    Then, by definition of $\psi_{abc}$, we have $\varphi_a(c) = \psi_{abc}(0)$, $\varphi_a(b) = \psi_{abc}(1)$ and $\psi_{abc}' = \dotprod{\nabla \varphi_a(c + t( b - c))}{b - c}$. Therefore, using the Newton-Leibniz formula, we have:
    \begin{align}
        \varphi_a(b) - \varphi_a(c) &= \psi_{abc}(1) - \psi_{abc}(0) = \int_{0}^1 \psi_{abc}' dt  = \int_{0}^1 \dotprod{\nabla \varphi_a(c + t( b - c))}{b - c} dt \notag\\
        &= \dotprod{\nabla \varphi_a(c)}{b - c} + \int_{0}^1 \dotprod{\nabla \varphi_a(c + t( b - c)) - \nabla \varphi_a(c)}{b - c} dt \notag\\
        &\overset{\eqref{eq:scalar_product_bound}}{\leq} \dotprod{\nabla \varphi_a(c)}{b - c} + \int_{0}^1 \norms{\nabla \varphi_a(c + t( b - c)) - \nabla \varphi_a(c)} \norms{b - c} dt \notag\\
        &\overset{\eqref{eq:varphi_smooth}}{\leq} \dotprod{\nabla \varphi_a(c)}{b - c} + \int_{0}^1 \left(L_0 + L_1 \norms{\nabla f(c)}\right) \norms{b - c}^2 \cdot t \cdot dt \notag\\
        &= \dotprod{\nabla \varphi_a(c)}{b - c} + \frac{\left(L_0 + L_1 \norms{\nabla f(c)}\right)}{2} \norms{b - c}^2. \label{eq:djdkfndkfnj}
    \end{align}
    Let $b = c - \frac{1}{L_0 + L_1 \norms{\nabla f(c)}} \nabla \varphi_a (c)$ and assume that $\|a - c\| \leq \frac{1}{L_1}$, then we have
    \begin{align*}
        \|b - c\| &=  \frac{\|\nabla \varphi_a (c)\|}{L_0 + L_1 \norms{\nabla f(c)}} = \frac{\|\nabla f(c) - \nabla f(a)\|}{L_0 + L_1 \norms{\nabla f(c)}} \overset{\eqref{eq:L0_L1_smoothness}}{\leq} \|c-a\| \leq \frac{1}{L_1}, 
    \end{align*}
    meaning that for this choice of $c$ and $b$ we can apply \eqref{eq:djdkfndkfnj} and get:
    \begin{equation*}
        \varphi_a(b) - \varphi_a(c) \leq  - \frac{\norms{\nabla \varphi_a(c)}^2}{L_0 + L_1 \norms{\nabla f(c)}} + \frac{\norms{\nabla \varphi_a(c)}^2}{2 \left(L_0 + L_1 \norms{\nabla f(c)}\right)} = - \frac{\norms{\nabla \varphi_a(c)}^2}{2 \left(L_0 + L_1 \norms{\nabla f(c)}\right)}.
    \end{equation*}

    Using the fact that $a$ is an optimum for $\varphi_a(c)$ (since $\nabla \varphi_a (a) = 0$) and by definition of $\varphi_a(c)$ we obtain the following inequality:
    \begin{equation*}
        f(a) - \dotprod{\nabla f(a)}{a} \leq f(c) - \dotprod{\nabla f(a)}{c} - \frac{\norms{\nabla f(c) - \nabla f (a)}^2}{2 \left(L_0 + L_1 \norms{\nabla f(c)}\right)}.
    \end{equation*}
    Using the fact that this inequality is satisfied for any $a, c\in \mathbb{R}^d$ such that $\|a-c\|\leq \frac{1}{L_1}$, we take $a= y$ and $c = x$ and we get the original statement of the Lemma:
    \begin{equation*}
        \frac{\norms{\nabla f(y) - \nabla f(x)}^2}{2 \left(L_0 + L_1 \norms{\nabla f(x)}\right)} \leq f(x) - f(y) - \dotprod{\nabla f(y)}{x - y}.
    \end{equation*}
\end{proof}
We are now ready to present the proofs of the gradient norm monotonicity along the trajectories of the considered first-order methods.
\begin{lemma}\label{lem:GD}
    Let function $f$ satisfy Assumption~\ref{ass:L0_L1_smooth} ($(L_0,L_1)$-smoothness) and Assumption~\ref{ass:strongly_convex} (convexity, ${\mu=0}$), then for all $k\geq 0$ Algorithm~\ref{algo:GD} with $\eta_k = (L_0 + L_1\|\nabla f(x^k)\|)^{-1}$ satisfies
    \begin{equation*}
        \norms{\nabla f(x^{k+1})} \leq \norms{\nabla f(x^k)}.
    \end{equation*}
\end{lemma}
\begin{proof}
    We note that for GD with $\eta_k = (L_0 + L_1\|\nabla f(x^k)\|)^{-1}$ iterates $x^k$ and $x^{k+1}$ satisfy  
    \begin{equation*}
        \|x^k - x^{k+1}\| = \frac{\|\nabla f(x^k)\|}{L_0 + L_1\|\nabla f(x^k)\|} \leq \frac{1}{L_1},
    \end{equation*}
    meaning that one can apply Lemma~\ref{lem:auxiliary_result} for these points. Introducing for convenience the new notation $\omega_k = L_0 + L_1 \norms{\nabla f(x)}$ and summing \eqref{eq:L0_L1_cocoercivity} with $x = x^k, y = x^{k+1}$ and $x = x^{k+1}, y = x^k$, we get the following inequality:
    \begin{align*}
        \left( \frac{1}{2 \omega_k} + \frac{1}{2 \omega_{k+1}} \right) \norms{\nabla f(x^{k+1}) - \nabla f(x^k)}^2 &\leq \dotprod{\nabla f(x^{k+1}) - \nabla f(x^k)}{x^{k+1} - x^k} \\
        &= - \eta_k \dotprod{\nabla f(x^{k+1}) - \nabla f(x^k)}{\nabla f(x^k)}.
    \end{align*}
    Multiplying both sides by $2 \omega_k$, we obtain 
    \begin{align*}
        \left( 1 + \frac{\omega_k}{\omega_{k+1}} \right) &\left(\norms{\nabla f(x^{k+1})}^2 -2\dotprod{\nabla f(x^{k+1})}{\nabla f(x^k)} +  \norms{\nabla f(x^k)}^2 \right) \\
        &\leq -2 \omega_k \eta_k \dotprod{\nabla f(x^{k+1}) - \nabla f(x^k)}{\nabla f(x^k)},
    \end{align*}
    which is equivalent to
    \begin{align*}
        \left( 1 + \frac{\omega_k}{\omega_{k+1}} \right) \norms{\nabla f(x^{k+1})}^2 &\leq \left( 1 + \frac{\omega_k}{\omega_{k+1}} \right) \norms{\nabla f(x^k)}^2\\
        &\quad + 2\left( 1 + \frac{ \omega_k}{\omega_{k+1}} - \omega_k \eta_k  \right) \dotprod{\nabla f(x^{k+1}) - \nabla f(x^k)}{\nabla f(x^k)} \\
        &= \left( 1 + \frac{\omega_k}{\omega_{k+1}} \right) \norms{\nabla f(x^k)}^2\\
        &\quad - \frac{2}{\eta_k}\left( 1 + \frac{ \omega_k}{\omega_{k+1}} - \omega_k \eta_k  \right) \dotprod{\nabla f(x^{k+1}) - \nabla f(x^k)}{x^{k+1} - x^k} \\
        &\overset{\circledOne}{=} \left( 1 + \frac{\omega_k}{\omega_{k+1}} \right) \norms{\nabla f(x^k)}^2 - \frac{2\omega_k}{\omega_{k+1} \eta_k}  \dotprod{\nabla f(x^{k+1}) - \nabla f(x^k)}{x^{k+1} - x^k} \\
        &\overset{\circledTwo}{\leq} \left( 1 + \frac{\omega_k}{\omega_{k+1}} \right) \norms{\nabla f(x^k)}^2,
    \end{align*}
    where in $\circledOne$ we used $\eta_k = \frac{1}{\omega_k}$; and in $\circledTwo$ we used $\eta_k,\omega_k, \omega_{k+1} \geq 0$ and convexity of function $f$. Hence, we obtain the original statement of the Lemma:
    \begin{equation*}
        \norms{\nabla f(x^{k+1})} \leq \norms{\nabla f(x^k)}.
    \end{equation*}
\end{proof}

Next, we provide a similar result for Algorithm~\ref{algo:NGD}.
\begin{lemma}\label{lem:NGD}
    Let function $f$ satisfy Assumption~\ref{ass:L0_L1_smooth} ($(L_0,L_1)$-smoothness) and Assumption~\ref{ass:strongly_convex} (convexity, ${\mu=0}$), then for all $k\geq 0$ Algorithm~\ref{algo:NGD} with $\eta_k = \eta \leq \frac{c}{L_0 + cL_1}$, where $\|\nabla f(x^k)\| \geq c$, satisfies
    \begin{equation*}
        \norms{\nabla f(x^{k+1})} \leq \norms{\nabla f(x^k)}.
    \end{equation*}
\end{lemma}
\begin{proof}
    We note that for NGD with $\eta_k = \eta \leq \frac{c}{L_0 + cL_1}$ iterates $x^k$ and $x^{k+1}$ satisfy  
    \begin{equation*}
        \|x^k - x^{k+1}\| = \eta \leq \frac{1}{L_1},
    \end{equation*}
    meaning that one can apply Lemma~\ref{lem:auxiliary_result} for these points. Introducing for convenience the new notation $\omega_k = L_0 + L_1 \norms{\nabla f(x)}$ and summing \eqref{eq:L0_L1_cocoercivity} with $x = x^k, y = x^{k+1}$ and $x = x^{k+1}, y = x^k$, we get the following inequality:
    \begin{align*}
        \left( \frac{1}{2 \omega_k} + \frac{1}{2 \omega_{k+1}} \right) \norms{\nabla f(x^{k+1}) - \nabla f(x^k)}^2 &\leq \dotprod{\nabla f(x^{k+1}) - \nabla f(x^k)}{x^{k+1} - x^k} \\
        &= - \frac{\eta_k}{\norms{\nabla f(x^k)}} \dotprod{\nabla f(x^{k+1}) - \nabla f(x^k)}{\nabla f(x^k)}.
    \end{align*}
    Multiplying both sides by $2 \omega_k$, we obtain 
    \begin{align*}
        \left( 1 + \frac{\omega_k}{\omega_{k+1}} \right)& \left(\norms{\nabla f(x^{k+1})}^2 -2\dotprod{\nabla f(x^{k+1})}{\nabla f(x^k)} +  \norms{\nabla f(x^k)}^2 \right)\\
        &\leq - \frac{2 \omega_k \eta_k}{\norms{\nabla f(x^k)}} \dotprod{\nabla f(x^{k+1}) - \nabla f(x^k)}{\nabla f(x^k)},
    \end{align*}
    which is equivalent to
    \begin{align*}
        \left( 1 + \frac{\omega_k}{\omega_{k+1}} \right) &\norms{\nabla f(x^{k+1})}^2 \leq \left( 1 + \frac{\omega_k}{\omega_{k+1}} \right) \norms{\nabla f(x^k)}^2 \\
        &\quad + 2\left( 1 + \frac{ \omega_k}{\omega_{k+1}} - \frac{ \omega_k \eta_k}{\norms{\nabla f(x^k)}}  \right) \dotprod{\nabla f(x^{k+1}) - \nabla f(x^k)}{\nabla f(x^k)} \\
        &= \left( 1 + \frac{\omega_k}{\omega_{k+1}} \right) \norms{\nabla f(x^k)}^2 \\
        &\quad - \frac{2 \| \nabla f(x^k) \|}{\eta_k}\left( 1 + \frac{ \omega_k}{\omega_{k+1}} - \frac{ \omega_k \eta_k}{\norms{\nabla f(x^k)}}  \right) \dotprod{\nabla f(x^{k+1}) - \nabla f(x^k)}{x^{k+1} - x^k} \\
        &\overset{\circledOne}{\leq} \left( 1 + \frac{\omega_k}{\omega_{k+1}} \right) \norms{\nabla f(x^k)}^2 \\
        &\quad - \frac{2 \| \nabla f(x^k) \|}{\eta_k}\left( 1 + \frac{ \omega_k}{\omega_{k+1}} - \frac{ \omega_k c}{\norms{\nabla f(x^k)} (L_0 + L_1c)}  \right) \dotprod{\nabla f(x^{k+1}) - \nabla f(x^k)}{x^{k+1} - x^k} \\
        &\overset{\circledTwo}{\leq} \left( 1 + \frac{\omega_k}{\omega_{k+1}} \right) \norms{\nabla f(x^k)}^2 - \frac{2\omega_k \norms{\nabla f(x^k)}}{\omega_{k+1} \eta_k}  \dotprod{\nabla f(x^{k+1}) - \nabla f(x^k)}{x^{k+1} - x^k} \\
        &\overset{\circledThree}{\leq} \left( 1 + \frac{\omega_k}{\omega_{k+1}} \right) \norms{\nabla f(x^k)}^2,
    \end{align*}
    where in $\circledOne$ we used $\eta_k \leq \frac{c}{L_0 + L_1 c}$, in $\circledTwo$ we used $\|\nabla f(x^k)\| \geq c$ implying $\frac{c}{L_0 + L_1 c} \leq \frac{\norms{\nabla f(x^k)}}{\omega_k}$, and in $\circledThree$ we used $\norms{\nabla f(x^k)}, \eta_k,\omega_k, \omega_{k+1} \geq 0$ and convexity of function $f$. Hence, we obtain the original statement of the Lemma:
    \begin{equation*}
        \norms{\nabla f(x^{k+1})} \leq \norms{\nabla f(x^k)}.
    \end{equation*}
\end{proof}

Finally, we present a similar result for Algorithm~\ref{algo:Clip-GD} that can be viewed as a combination of the previous two.
\begin{lemma}\label{lem:ClipGD}
    Let function $f$ satisfy Assumption~\ref{ass:L0_L1_smooth} ($(L_0,L_1)$-smoothness) and Assumption~\ref{ass:strongly_convex} (convexity, ${\mu=0}$), then for all $k\geq 0$ Algorithm~\ref{algo:Clip-GD} with step size $\eta_k = (L_0 + L_1 \max\{ \| \nabla f(x^k) \|,c\})^{-1}$ satisfies
    \begin{equation*}
        \norms{\nabla f(x^{k+1})} \leq \norms{\nabla f(x^k)}.
    \end{equation*}
\end{lemma}
\begin{proof}
    We note that for Clip-GD with $\eta_k = (L_0 + L_1 \max\{ \| \nabla f(x^k) \|,c\})^{-1}$ iterates $x^k$ and $x^{k+1}$ satisfy  
    \begin{equation*}
        \|x^k - x^{k+1}\| = \frac{\max\{ \| \nabla f(x^k) \|,c\}}{L_0 + L_1 \max\{ \| \nabla f(x^k) \|,c\}} \leq \frac{1}{L_1},
    \end{equation*}
    meaning that one can apply Lemma~\ref{lem:auxiliary_result} for these points. Introducing for convenience the new notation $\omega_k = L_0 + L_1 \norms{\nabla f(x)}$ and summing \eqref{eq:L0_L1_cocoercivity} with $x = x^k, y = x^{k+1}$ and $x = x^{k+1}, y = x^k$, we get the following inequality:
    \begin{align*}
        \left( \frac{1}{2 \omega_k} + \frac{1}{2 \omega_{k+1}} \right) \norms{\nabla f(x^{k+1}) - \nabla f(x^k)}^2 &\leq \dotprod{\nabla f(x^{k+1}) - \nabla f(x^k)}{x^{k+1} - x^k} \\
        &= - \eta_k \cdot \underbrace{\min \left\{ 1, \frac{c}{\norms{\nabla f(x^k)}}\right\}}_{\lambda_k} \dotprod{\nabla f(x^{k+1}) - \nabla f(x^k)}{\nabla f(x^k)}.
    \end{align*}
    Multiplying both sides by $2 \omega_k$, we obtain 
    \begin{align} 
        \left( 1 + \frac{\omega_k}{\omega_{k+1}} \right) &\left(\norms{\nabla f(x^{k+1})}^2 -2\dotprod{\nabla f(x^{k+1})}{\nabla f(x^k)} +  \norms{\nabla f(x^k)}^2 \right) \notag\\
        &\leq - 2 \omega_k \eta_k \lambda_k \dotprod{\nabla f(x^{k+1}) - \nabla f(x^k)}{\nabla f(x^k)}.\label{eq:aux_clip}
    \end{align}
    Consider two cases: $\lambda_k = 1$ or $\lambda_k = \frac{c}{\|\nabla f(x^k)\|}$. If $\lambda_k = 1$, then $c \geq \norms{\nabla f(x^k)}$. Then \eqref{eq:aux_clip} is equivalent to the following:
    \begin{align*}
        \left( 1 + \frac{\omega_k}{\omega_{k+1}} \right) &\norms{\nabla f(x^{k+1})}^2 \leq \left( 1 + \frac{\omega_k}{\omega_{k+1}} \right) \norms{\nabla f(x^k)}^2 \\
        &\quad + 2\left( 1 + \frac{ \omega_k}{\omega_{k+1}} - \omega_k \eta_k  \right) \dotprod{\nabla f(x^{k+1}) - \nabla f(x^k)}{\nabla f(x^k)} \\
        &= \left( 1 + \frac{\omega_k}{\omega_{k+1}} \right) \norms{\nabla f(x^k)}^2 - \frac{2}{\eta_k}\left( 1 + \frac{ \omega_k}{\omega_{k+1}} - \omega_k \eta_k  \right) \dotprod{\nabla f(x^{k+1}) - \nabla f(x^k)}{x^{k+1} - x^k} \\
        &\overset{\circledOne}{\leq} \left( 1 + \frac{\omega_k}{\omega_{k+1}} \right) \norms{\nabla f(x^k)}^2\\
        &\quad - \frac{2}{\eta_k}\left( 1 + \frac{ \omega_k}{\omega_{k+1}} - \frac{\omega_k}{L_0 + L_1 c}  \right) \dotprod{\nabla f(x^{k+1}) - \nabla f(x^k)}{x^{k+1} - x^k} \\
        &\overset{\circledTwo}{\leq} \left( 1 + \frac{\omega_k}{\omega_{k+1}} \right) \norms{\nabla f(x^k)}^2 - \frac{2\omega_k}{\omega_{k+1} \eta_k}  \dotprod{\nabla f(x^{k+1}) - \nabla f(x^k)}{x^{k+1} - x^k} \\
        &\overset{\circledThree}{\leq} \left( 1 + \frac{\omega_k}{\omega_{k+1}} \right) \norms{\nabla f(x^k)}^2,
    \end{align*}
    where in $\circledOne$ we used $\eta_k \leq \frac{1}{L_0 + L_1 c}$, in $\circledTwo$ we used $\frac{1}{L_0 + L_1 c} \leq \frac{1}{\omega_k}$, and in $\circledThree$ we used $\eta_k,\omega_k, \omega_{k+1} \geq 0$ and convexity of function $f$. 
    
    Next, we consider the case when $\lambda_k = \frac{c}{\norms{\nabla f(x^k)}}$, implying $c \leq \norms{\nabla f(x^k)}$. Then, \eqref{eq:aux_clip} is equivalent to the following:
    \begin{align*}
        \left( 1 + \frac{\omega_k}{\omega_{k+1}} \right) &\norms{\nabla f(x^{k+1})}^2 \leq \left( 1 + \frac{\omega_k}{\omega_{k+1}} \right) \norms{\nabla f(x^k)}^2 \\
        &\quad + 2\left( 1 + \frac{ \omega_k}{\omega_{k+1}} - \frac{ \omega_k \eta_k c}{\norms{\nabla f(x^k)}}  \right) \dotprod{\nabla f(x^{k+1}) - \nabla f(x^k)}{\nabla f(x^k)} \\
        &= \left( 1 + \frac{\omega_k}{\omega_{k+1}} \right) \norms{\nabla f(x^k)}^2 \\
        &\quad - \frac{2 \| \nabla f(x^k) \|}{\eta_k c}\left( 1 + \frac{ \omega_k}{\omega_{k+1}} - \frac{ \omega_k \eta_k c}{\norms{\nabla f(x^k)}}  \right) \dotprod{\nabla f(x^{k+1}) - \nabla f(x^k)}{x^{k+1} - x^k} \\
        &\overset{\circledOne}{\leq} \left( 1 + \frac{\omega_k}{\omega_{k+1}} \right) \norms{\nabla f(x^k)}^2 \\
        &\quad - \frac{2 \| \nabla f(x^k) \|}{\eta_k c}\left( 1 + \frac{ \omega_k}{\omega_{k+1}} - \frac{ \omega_k}{(L_0 + L_1c)}  \right) \dotprod{\nabla f(x^{k+1}) - \nabla f(x^k)}{x^{k+1} - x^k} \\
        &\overset{\circledTwo}{\leq} \left( 1 + \frac{\omega_k}{\omega_{k+1}} \right) \norms{\nabla f(x^k)}^2 - \frac{2\omega_k \norms{\nabla f(x^k)}}{\omega_{k+1} \eta_k c}  \dotprod{\nabla f(x^{k+1}) - \nabla f(x^k)}{x^{k+1} - x^k} \\
        &\overset{\circledThree}{\leq} \left( 1 + \frac{\omega_k}{\omega_{k+1}} \right) \norms{\nabla f(x^k)}^2,
    \end{align*}
    where in $\circledOne$ we used $\eta_k \leq \frac{1}{L_0 + L_1 c}$ and $c \leq \norms{\nabla f(x^k)}$, in $\circledTwo$ we used $\frac{c}{L_0 + L_1 c} \leq \frac{1}{\omega_k}$, and in $\circledThree$ we used $\norms{\nabla f(x^k)}, \eta_k,\omega_k, \omega_{k+1} \geq 0$ and convexity of function $f$.
    
    That is, in both cases, we obtain the original statement of the Lemma:
    \begin{equation*}
        \norms{\nabla f(x^{k+1})} \leq \norms{\nabla f(x^k)}.
    \end{equation*}
\end{proof}

\section{Missing Proofs for Full-Gradient Algorithms}
In this section, we give missing proofs from the main part of the paper. In particular, see Subsection~\ref{subsec:Proof_GD} for the proof of convergence results for Algorithm~\ref{algo:GD}, see Subsection~\ref{subsec:Proof_NGD} for Algorithm~\ref{algo:NGD}, and see Subsection~\ref{subsec:Proof_ClipGD} for Algorithm~\ref{algo:Clip-GD}.

\subsection{Proof of Theorem~\ref{th:GD_convex}}\label{subsec:Proof_GD}

Using Assumption~\ref{ass:L0_L1_smooth}, we derive
\begin{align}
    f(x^{k+1}) - f(x^k) &= f(x^{k} - \eta_k \nabla f(x^k)) - f(x^k) \nonumber \\
    &\overset{\eqref{eq:ass_smooth}}{\leq} - \eta_k \dotprod{\nabla f(x^k)}{\nabla f(x^k)} + \eta_k^2 \frac{L_0 + L_1 \norms{\nabla f(x^k)}}{2} \norms{\nabla f(x^k)}^2 \nonumber \\
    & \overset{\circledOne}{\leq} - \eta_k \norms{\nabla f(x^k)}^2 + \frac{\eta_k}{2} \norms{\nabla f(x^k)}^2 \nonumber \\
    &= - \frac{\eta_k}{2} \norms{\nabla f(x^k)}^2, \label{eq:GD_convex_1}
\end{align}
where in $\circledOne$ we used $\eta_k \leq \frac{1}{L_0 + L_1 \norms{\nabla f(x^k)}}$. Next, let us consider two cases.

\fbox{The case of $\norms{\nabla f(x^k)} \geq \frac{L_0}{L_1}$.} Taking $\eta_k = \frac{1}{L_0 + L_1 \norms{\nabla f(x^k)}}$ and using the convexity assumption of the function (see Assumption~\ref{ass:strongly_convex}, $\mu = 0$), we have the following:
\begin{align*}
    f(x^k) - f^* &\leq \dotprod{\nabla f(x^k)}{x^k - x^*}\overset{\eqref{eq:scalar_product_bound}}{\leq} \norms{\nabla f(x^k)} \norms{x^k - x^*} \overset{\circledOne}{\leq} \norms{\nabla f(x^k)} \underbrace{\norms{x^0 - x^*}}_{R} = \frac{\eta_k}{\eta_k} \norms{\nabla f(x^k)} R\\
    &= \eta_k (L_0 + L_1 \norms{\nabla f(x^k)}) \norms{\nabla f(x^k)} R \leq 2 \eta_k L_1 \norms{\nabla f(x^k)}^2 R,
\end{align*}
where $\circledOne$ follows from $\|x^k - x^*\| \leq \|x^0 - x^*\|$ \citep[proof of Theorem 3.3]{Gorbunov_2024}. The above inequality implies
\begin{equation}
    \label{eq:GD_case1_2}
    \eta_k \geq \frac{f(x^k) - f^*}{2 L_1 R \norms{\nabla f(x^k)}^2}.
\end{equation}
Plugging \eqref{eq:GD_case1_2} into \eqref{eq:GD_convex_1}, we obtain
\begin{equation*}
    f(x^{k+1}) - f(x^k) \leq - \eta_k \norms{\nabla f(x^k)}^2 \leq \frac{1}{4 L_1 R} (f(x^k) - f^*),
\end{equation*}
which is equivalent to
\begin{equation}
    f(x^{k+1}) - f^* \leq \left( 1 - \frac{1}{4 L_1 R} \right) \left( f(x^k) - f^* \right). \label{eq:bdvnjkdfkvf}
\end{equation}
Moreover, Lemma~\ref{lem:GD} implies that for all $t = 0, \ldots, k$ a similar inequality holds. We denote $T \coloneqq \min \left\{ k \in \{0,1,2,...,N-1\}\;\;\; | \;\;\;\norms{\nabla f(x^k)} < \frac{L_0}{L_1} \text{ and } \norms{\nabla f(x^{k-1})} \geq \frac{L_0}{L_1}\right\}$ as the first index $k$ such that $\norms{\nabla f(x^k)} < \frac{L_0}{L_1}$ (note that $T = 0$ is possible). Then, for the first $T$ iterations, we have linear convergence:
\begin{equation}
    f(x^T) - f^* \leq \left( 1 - \frac{1}{4 L_1 R} \right)^T \left( f(x^0) - f^* \right), \label{eq:GD_rate_1}
\end{equation}
which follows from unrolling \eqref{eq:bdvnjkdfkvf}.

\fbox{The case of $\norms{\nabla f(x^k)} <\frac{L_0}{L_1}$.} Taking $\eta_k = \frac{1}{L_0 + L_1 \norms{\nabla f(x^k)}}$ and
using the convexity assumption of the function (see Assumption~\ref{ass:strongly_convex}, $\mu = 0$), we have the following:
\begin{align}
    f(x^k) - f^* &\leq \dotprod{\nabla f(x^k)}{x^k - x^*} \overset{\eqref{eq:scalar_product_bound}}{\leq} \norms{\nabla f(x^k)} \norms{x^k - x^*} \leq \norms{\nabla f(x^k)} \underbrace{\norms{x^0 - x^*}}_{R} \label{eq:jdfjdfjhdhfj}\\
    &= \frac{\eta_k}{\eta_k} \norms{\nabla f(x^k)} R = \eta_k (L_0 + L_1 \norms{\nabla f(x^k)}) \norms{\nabla f(x^k)} R < 2 \eta_k L_0 \norms{\nabla f(x^k)} R. \notag
\end{align}
The above inequality implies
\begin{equation}
    \label{eq:GD_case2_2}
    \eta_k > \frac{f(x^k) - f^*}{2 L_0 R \norms{\nabla f(x^k)}}.
\end{equation}
Then, plugging \eqref{eq:GD_case2_2} into \eqref{eq:GD_convex_1} and using the notation $F_{k} = f(x^k) - f^*$, we obtain:
\begin{equation*}
    F_{k+1} < F_k - \frac{\norms{\nabla f(x^k)}}{4 L_0 R} F_k \overset{\eqref{eq:jdfjdfjhdhfj}}{\leq} F_k - \frac{1}{4 L_0 R^2} F_k^2,
\end{equation*}
which is equivalent to
\begin{equation*}
    \frac{1}{4  L_0 R^2} F_k^2 < F_k - F_{k+1}.
\end{equation*}
Next, we divide both sides by $F_{k+1} F_k$
\begin{equation*}
    \frac{1}{4  L_0 R^2} \cdot \frac{F_k}{F_{k+1}}  <  \frac{1}{F_{k+1}} - \frac{1}{F_{k}}
\end{equation*}
and use that $F_{k+1} \leq F_k$ due to \eqref{eq:GD_convex_1}:
\begin{equation*}
    \frac{1}{4  L_0 R^2}  <  \frac{1}{F_{k+1}} - \frac{1}{F_{k}}.
\end{equation*}
Summing up the above inequality for $k=T, T+1, ..., N$, we get
\begin{equation*}
   \frac{N-T}{4  L_0 R^2} = \sum_{k = T}^{N-1} \frac{1}{4  L_0 R^2}  <  \sum_{k = T}^{N-1} \left(\frac{1}{F_{k+1}} - \frac{1}{F_{k}}\right) = \frac{1}{F_{N}} - \frac{1}{F_{T}} < \frac{1}{F_{N}},
\end{equation*}
which is equivalent to
\begin{equation}
    f(x^N) - f^* < \frac{4L_0R^2}{N-T}. \label{eq:GD_rate_2}
\end{equation}

Finally, combining inequalities \eqref{eq:GD_rate_1} and \eqref{eq:GD_rate_2} and taking into account that $F_N \leq F_T$, we obtain the convergence rate of Algorithm~\ref{algo:GD} in the convex case:
\begin{equation*}
    f(x^N) - f^* = \mathcal{O} \left( \min\left\{\frac{L_0R^2}{N-T}, \left( 1 - \frac{1}{L_1 R} \right)^T F_0\right\} \right),
\end{equation*}
where $T \coloneqq \min \left\{ k \in \{0,1,2,...,N-1\}\;\;\; | \;\;\;\norms{\nabla f(x^k)} < \frac{L_0}{L_1} \text{ and } \norms{\nabla f(x^{k-1})} \geq \frac{L_0}{L_1}\right\}$.

\subsection{Proof of Theorem~\ref{th:NGD}}\label{subsec:Proof_NGD}

Using Assumption~\ref{ass:L0_L1_smooth}, we derive
\begin{align}
    f(x^{k+1}) - f(x^k) &= f\left(x^{k} - \eta_k \frac{\nabla f(x^k)}{\norms{\nabla f(x^k)}}\right) - f(x^k) \nonumber \\
    &\overset{\eqref{eq:ass_smooth}}{\leq} - \frac{\eta_k}{\norms{\nabla f(x^k)}} \dotprod{\nabla f(x^k)}{\nabla f(x^k)} + \eta_k^2 \frac{L_0 + L_1 \norms{\nabla f(x^k)}}{2 \norms{\nabla f(x^k)}^2} \norms{\nabla f(x^k)}^2 \nonumber \\
    & \overset{\circledOne}{\leq} - \eta_k \norms{\nabla f(x^k)} + \frac{\eta_k}{2} \norms{\nabla f(x^k)} \nonumber \\
    &= - \frac{\eta_k}{2} \norms{\nabla f(x^k)}, \label{eq:NGD_convex_1}
\end{align}
where in $\circledOne$ we used $\eta_k = \eta \leq \frac{c}{L_0 + L_1c} \leq \frac{\norms{\nabla f(x^k)}}{L_0 + L_1 \norms{\nabla f(x^k)}}$ since $\norms{\nabla f(x^k)} \geq c$ for all $k = 0,1,\ldots, N-1$ and function $\varphi(u) = \frac{u}{L_0 + L_1 u}$ is increasing function in $u \geq 0$.

Next, we us the convexity assumption of the function (see Assumption~\ref{ass:strongly_convex}, $\mu = 0$):
\begin{align}
    f(x^k) - f^* &\leq \dotprod{\nabla f(x^k)}{x^k - x^*}\overset{\eqref{eq:scalar_product_bound}}{\leq} \norms{\nabla f(x^k)} \norms{x^k - x^*} \overset{\circledOne}{\leq} \norms{\nabla f(x^k)} \underbrace{\norms{x^0 - x^*}}_{R}, \label{eq:chudhfucdsybcshj}
\end{align}
where $\circledOne$ follows from $\|x^k - x^*\| \leq \|x^0 - x^*\|$:
\begin{align*}
    \|x^k - x^*\|^2 &= \|x^{k-1} - x^*\|^2 - \frac{2\eta_k}{\|\nabla f(x^k)\|} \langle \nabla f(x^k), x^k - x^* \rangle + \eta_k^2\\
    &\overset{\eqref{eq:str_cvx}}{\leq}  \|x^{k-1} - x^*\|^2 - \frac{2\eta(f(x^k) - f^*)}{\|\nabla f(x^k)\|}  + \eta^2\\
    &= \|x^{k-1} - x^*\|^2 - \eta\left(\frac{2(f(x^k) - f^*)}{\|\nabla f(x^k)\|} - \eta\right)\\
    &\leq \|x^{k-1} - x^*\|^2,
\end{align*}
where in the last step, we use
\begin{equation*}
    \frac{\eta\|\nabla f(x^k)\|}{2} \leq \frac{c\|\nabla f(x^k)\|}{2(L_0 + L_1c)} \leq \frac{\norms{\nabla f(x^k)}^2}{2(L_0 + L_1 \norms{\nabla f(x^k)})} \overset{\eqref{eq:L0_L1_cocoercivity}}{\leq} f(x^k) - f^*.
\end{equation*}
Next, inequality \eqref{eq:chudhfucdsybcshj} gives
\begin{equation}
    \label{eq:NGD_convex_2}
    \norms{\nabla f(x^k)} \geq \frac{f(x^k) - f^*}{R}.
\end{equation}
Then, plugging \eqref{eq:NGD_convex_2} into \eqref{eq:NGD_convex_1}, we obtain:
\begin{equation*}
    f(x^{k+1}) - f(x^k) \leq - \frac{\eta}{2} \norms{\nabla f(x^k)} \leq -\frac{\eta}{2 R} (f(x^k) - f^*),
\end{equation*}
which is equivalent to
\begin{equation*}
    f(x^{k+1}) - f^* \leq \left( 1 - \frac{\eta}{2 R} \right) \left( f(x^k) - f^* \right).
\end{equation*}
Unrolling the above recurrence, we derive the linear convergence for NGD with step size $\eta_k = \eta \leq \frac{c}{L_0 + L_1c}$:
\begin{equation*}
    f(x^N) - f^* \leq \left( 1 - \frac{\eta}{2 R} \right)^N \left( f(x^0) - f^* \right).
\end{equation*}

\subsection{Proof of Theorem~\ref{th:Clip-GD}}\label{subsec:Proof_ClipGD}

Since $\lambda_k = \min \left\{1, \frac{c}{\norms{\nabla f(x^k)}}  \right\}$, we there are only two possible cases for $\lambda_k$: either $\lambda_k = 1$ or $\lambda_k = \frac{c}{\norms{\nabla f(x^k)}}
$.
\begin{enumerate}
    \item[i)] Consider the case of $\lambda_k = \frac{c}{\norms{\nabla f(x^k)}}$, i.e., $c \leq \|\nabla f(x^k)\|$. Using Assumption~\ref{ass:L0_L1_smooth}, we derive
\begin{align}
    f(x^{k+1}) - f(x^k) &= f(x^{k} - \eta_k \lambda_k \nabla f(x^k)) - f(x^k) \nonumber \\
    &= f\left(x^{k} - \eta_k \frac{c}{\norms{\nabla f(x^k)}} \nabla f(x^k)\right) - f(x^k) \nonumber \\
    &\overset{\eqref{eq:ass_smooth}}{\leq} - \eta_k \frac{c}{\norms{\nabla f(x^k)}} \dotprod{\nabla f(x^k)}{\nabla f(x^k)} \notag \\
    &\quad + \eta_k^2 \frac{c^2}{\norms{\nabla f(x^k)}^2} \frac{L_0 + L_1 \norms{\nabla f(x^k)}}{2} \norms{\nabla f(x^k)}^2 \nonumber \\
    &= - \eta_k c \norms{\nabla f(x^k)} + \eta_k^2 c^2 \frac{L_0 + L_1 \norms{\nabla f(x^k)}}{2} \nonumber \\
    & \overset{\circledOne}{\leq} - \eta_k c \norms{\nabla f(x^k)} + \frac{\eta_k c}{2} \norms{\nabla f(x^k)} \nonumber \\
    &= -\frac{\eta_k c}{2} \norms{\nabla f(x^k)}, \label{eq:ClipGD1}
\end{align}
where in $\circledOne$ we used $\eta_k \leq \frac{\norms{\nabla f(x^k)}}{c(L_0 + L_1 \norms{\nabla f(x^k)})}$, which follows from $c \leq \norms{\nabla f(x^k)}$:
\begin{align*}
    \frac{\norms{\nabla f(x^k)}}{c(L_0 + L_1 \norms{\nabla f(x^k)})} = \frac{1}{L_0 \frac{c}{\norms{\nabla f(x^k)}} + L_1 c} \geq \frac{1}{L_0 + L_1 c} = \eta = \eta_k.
\end{align*}

Next, using the convexity assumption of the function (see Assumption~\ref{ass:strongly_convex}, $\mu = 0$), we get
\begin{align}
    f(x^k) - f^* &\leq \dotprod{\nabla f(x^k)}{x^k - x^*} \overset{\eqref{eq:scalar_product_bound}}{\leq} \norms{\nabla f(x^k)} \norms{x^k - x^*} \overset{\circledOne}{\leq} \norms{\nabla f(x^k)} \underbrace{\norms{x^0 - x^*}}_{R}, \label{eq:hdjbfdhjfvdhfvjdf}
\end{align}
where $\circledOne$ follows from $\|x^k - x^*\| \leq \|x^0 - x^*\|$:
\begin{align*}
    \|x^k - x^*\|^2 &= \|x^{k-1} - x^*\|^2 - \frac{2c\eta_k}{\|\nabla f(x^k)\|} \langle \nabla f(x^k), x^k - x^* \rangle + c^2\eta_k^2\\
    &\overset{\eqref{eq:str_cvx}}{\leq}  \|x^{k-1} - x^*\|^2 - \frac{2c\eta(f(x^k) - f^*)}{\|\nabla f(x^k)\|}  + c^2\eta^2\\
    &= \|x^{k-1} - x^*\|^2 - c\eta\left(\frac{2(f(x^k) - f^*)}{\|\nabla f(x^k)\|} - c\eta\right)\\
    &\leq \|x^{k-1} - x^*\|^2,
\end{align*}
where in the last step, we use
\begin{equation*}
    \frac{c\eta\|\nabla f(x^k)\|}{2} \leq \frac{c\|\nabla f(x^k)\|}{2(L_0 + L_1c)} \leq \frac{\norms{\nabla f(x^k)}^2}{2(L_0 + L_1 \norms{\nabla f(x^k)})} \overset{\eqref{eq:L0_L1_cocoercivity}}{\leq} f(x^k) - f^*.
\end{equation*}
Inequality \eqref{eq:hdjbfdhjfvdhfvjdf} gives
\begin{equation}
    \label{eq:ClipGD1_2}
    \norms{\nabla f(x^k)} \geq \frac{f(x^k) - f^*}{R}.
\end{equation}
Then, plugging \eqref{eq:ClipGD1_2} into \eqref{eq:ClipGD1}, we obtain
\begin{equation*}
    f(x^{k+1}) - f(x^k) \overset{\eqref{eq:ClipGD1}}{\leq} - \frac{\eta c}{2} \norms{\nabla f(x^k)} \leq \frac{\eta c}{2 R} (f(x^k) - f^*),
\end{equation*}
which is equivalent to
\begin{equation*}
    f(x^{k+1}) - f^* \leq  \left( 1 - \frac{\eta c}{2 R} \right) \left( f(x^k) - f^* \right).
\end{equation*}
Next, we consider two possible scenarios for the convergence of the algorithm depending on the relation between $\norms{\nabla f(x^k)}, c$ and $\frac{L_0}{L_1}$ (note that $\norms{\nabla f(x^k)} \geq c$ in this case), given the monotonicity of the gradient norm (Lemma~\ref{lem:ClipGD}).
\begin{itemize}
    \item[$(\mathcal{T})$] If for $k = 0, 1, 2, ..., \mathcal{T}_1 - 1$, the iterates of Clip-GD satisfy $\norms{\nabla f(x^k)} \geq c \geq \frac{L_0}{L_1}$, then $\eta \geq \frac{1}{2L_1 c}$ and we have linear convergence for the first $\mathcal{T}_1$ iterations:
    \begin{equation}
        f(x^{\mathcal{T}_1}) - f^* \leq \left( 1 - \frac{1}{4 L_1 R} \right)^{\mathcal{T}_1} \left( f(x^0) - f^* \right). \label{eq:clipGD1_scenari_1}
    \end{equation}

    \item[$(\mathcal{K})$] If for $k = 0, 1, 2, ..., \mathcal{K}_1 - 1$, the iterates of Clip-GD satisfy $\norms{\nabla f(x^k)} \geq \frac{L_0}{L_1} \geq c$ or $\frac{L_0}{L_1} \geq \norms{\nabla f(x^k)} \geq c$, then $\eta \geq \frac{1}{2L_0}$ and we have linear convergence of the first $\mathcal{K}_1$ iterations:
    \begin{equation}
        f(x^{\mathcal{K}_1}) - f^* \leq \left( 1 - \frac{c}{4 L_0 R} \right)^{\mathcal{K}_1} \left( f(x^0) - f^* \right). \label{eq:clipGD1_scenari_2}
    \end{equation}
\end{itemize}

 \item[ii)] Consider the case of $\lambda_k = 1$, i.e., $c \geq \|\nabla f(x^k)\|$. Using Assumption~\ref{ass:L0_L1_smooth}, we derive
\begin{align}
    f(x^{k+1}) - f(x^k) &= f(x^{k} - \eta_k \lambda_k \nabla f(x^k)) - f(x^k) \nonumber \\
    &= f(x^{k} - \eta_k \nabla f(x^k)) - f(x^k) \nonumber \\
    &\overset{\eqref{eq:ass_smooth}}{\leq} - \eta_k \dotprod{\nabla f(x^k)}{\nabla f(x^k)} + \eta_k^2 \frac{L_0 + L_1 \norms{\nabla f(x^k)}}{2} \norms{\nabla f(x^k)}^2 \nonumber \\
    &= - \eta_k \norms{\nabla f(x^k)}^2 + \eta_k^2  \frac{L_0 + L_1 \norms{\nabla f(x^k)}}{2} \norms{\nabla f(x^k)}^2 \nonumber \\
    & \overset{\circledOne}{=} - \frac{1}{2(L_0 + L_1 \norms{\nabla f(x^k)})}\norms{\nabla f(x^k)}^2, \label{eq:ClipGD2}
\end{align}
where in $\circledOne$ we used $\eta_k = \frac{1}{L_0 + L_1 \norms{\nabla f(x^k)}}$.  Using the convexity assumption of the function (see Assumption~\ref{ass:strongly_convex}, $\mu = 0$), we get
\begin{align}
    f(x^k) - f^* &\leq \dotprod{\nabla f(x^k)}{x^k - x^*} \overset{\eqref{eq:scalar_product_bound}}{\leq} \norms{\nabla f(x^k)} \norms{x^k - x^*} \overset{\circledOne}{\leq} \norms{\nabla f(x^k)} \underbrace{\norms{x^0 - x^*}}_{R}, \label{eq:nkdvfvjkdfjvkdvj}
\end{align}
where $\circledOne$ follows from $\|x^k - x^*\| \leq \|x^0 - x^*\|$:
\begin{align*}
    \|x^k - x^*\|^2 &= \|x^{k-1} - x^*\|^2 - 2\eta_k \langle \nabla f(x^k), x^k - x^* \rangle + \eta_k^2\|\nabla f(x^k)\|^2\\
    &\overset{\eqref{eq:str_cvx}}{\leq}  \|x^{k-1} - x^*\|^2 - 2\eta_k(f(x^k) - f^*) + \eta_k^2\|\nabla f(x^k)\|^2\\
    &= \|x^{k-1} - x^*\|^2 - \eta_k\left(2(f(x^k) - f^*) - \eta_k\|\nabla f(x^k)\|^2\right)\\
    &\leq \|x^{k-1} - x^*\|^2,
\end{align*}
where in the last step, we use
\begin{equation*}
    \frac{\eta_k\|\nabla f(x^k)\|}{2} \leq \frac{\norms{\nabla f(x^k)}^2}{2(L_0 + L_1 \norms{\nabla f(x^k)})} \overset{\eqref{eq:L0_L1_cocoercivity}}{\leq} f(x^k) - f^*.
\end{equation*}
Inequality \eqref{eq:nkdvfvjkdfjvkdvj}, implies
\begin{equation}
    \label{eq:ClipGD2_case1_2}
    \norms{\nabla f(x^k)} \geq \frac{f(x^k) - f^*}{ R}.
\end{equation}
Next, we consider two cases: $\norms{\nabla f(x^k)} \geq \frac{L_0}{L_1}$ and $\norms{\nabla f(x^k)} < \frac{L_0}{L_1}$.

\fbox{The case of $\norms{\nabla f(x^k)} \geq \frac{L_0}{L_1}$.} In this case, inequality \eqref{eq:ClipGD2} gives
\begin{align}
    f(x^{k+1}) - f(x^k) & \leq - \frac{1}{4L_1} \norms{\nabla f(x^k)}. \label{eq:ClipGD2_case1_1}
\end{align}

Then, plugging \eqref{eq:ClipGD2_case1_2} into \eqref{eq:ClipGD2_case1_1}, we obtain:
\begin{equation*}
    f(x^{k+1}) - f(x^k) \leq -\frac{1}{4 L_1 R} (f(x^k) - f^*),
\end{equation*}
which is equivalent to
\begin{equation*}
    f(x^{k+1}) - f^* \leq \left( 1 - \frac{1}{4 L_1 R} \right) \left( f(x^k) - f^* \right).
\end{equation*}
Since in this case we have the following relation $c \geq \norms{\nabla f(x^k)} \geq \frac{L_0}{L_1}$, then for $k = \mathcal{T}_1, \mathcal{T}_1 + 1, ..., \mathcal{T}_2 -1$ we have linear convergence:
\begin{equation}
    f(x^{\mathcal{T}_2}) - f^* \leq \left( 1 - \frac{1}{4 L_1 R} \right)^{\mathcal{T}_2 - \mathcal{T}_1} \left( f(x^{\mathcal{T}_1}) - f^* \right) \overset{\eqref{eq:clipGD1_scenari_1}}{\leq} \left( 1 - \frac{1}{4 L_1 R} \right)^{\mathcal{T}_2} \left( f(x^0) - f^* \right). \label{eq:clipGD2_scenari_1}
\end{equation}

\fbox{The case of $\norms{\nabla f(x^k)} < \frac{L_0}{L_1}$.} In this case, inequality \eqref{eq:ClipGD2} gives
\begin{align}
    f(x^{k+1}) - f(x^k) &\leq - \frac{1}{2(L_0 + L_1 \norms{\nabla f(x^k)})}\norms{\nabla f(x^k)}^2 \nonumber\\
    & < - \frac{1}{4L_0} \norms{\nabla f(x^k)}^2. \label{eq:ClipGD2_case2_1}
\end{align}

Then, plugging \eqref{eq:ClipGD2_case1_2} into \eqref{eq:ClipGD2_case2_1} and using the notation $F_{k} \coloneqq f(x^k) - f^*$, we obtain:
\begin{equation*}
    F_{k+1} < F_k - \frac{\norms{\nabla f(x^k)}}{4 L_0 R} F_k \leq F_k - \frac{1}{4 L_0 R^2} F_k^2,
\end{equation*}
which is equivalent to
\begin{equation*}
    \frac{1}{4  L_0 R^2} F_k^2 < F_k - F_{k+1}.
\end{equation*}
Next, we divide both sides by $F_{k+1} F_k$
\begin{equation*}
    \frac{1}{4  L_0 R^2} \cdot \frac{F_k}{F_{k+1}}  <  \frac{1}{F_{k+1}} - \frac{1}{F_{k}}.
\end{equation*}
and use $F_{k+1} \leq F_k$ due to \eqref{eq:ClipGD2_case2_1}:
\begin{equation}
    \frac{1}{4  L_0 R^2}  <  \frac{1}{F_{k+1}} - \frac{1}{F_{k}}. \label{eq:bdhjbfjhbdjbvdfvhdjv}
\end{equation}

Then, two situations are possible: either $\frac{L_0}{L_1} > c$ or $\frac{L_0}{L_1} \leq c$. We consider each of them separately.
\begin{itemize}
    \item[$(\mathcal{K})$] Considering the scenario $\frac{L_0}{L_1} > c > \norms{\nabla f(x^k)}$ and summing up inequality \eqref{eq:bdhjbfjhbdjbvdfvhdjv} for $k=\mathcal{K}_1, \mathcal{K}_1+1, ..., \mathcal{K}_2-1$, we get
    \begin{equation*}
        \frac{\mathcal{K}_2-\mathcal{K}_1}{4  L_0 R^2} = \sum_{k = \mathcal{K}_1}^{\mathcal{K}_2-1} \frac{1}{4  L_0 R^2}  <  \sum_{k = \mathcal{K}_1}^{\mathcal{K}_2-1} \left(\frac{1}{F_{k+1}} - \frac{1}{F_{k}}\right) = \frac{1}{F_{\mathcal{K}_2}} - \frac{1}{F_{\mathcal{K}_1}} < \frac{1}{F_{\mathcal{K}_2}},
    \end{equation*}
    which is equivalent to
    \begin{equation}\label{eq:clipGD2_scenari_2}
        f(x^{\mathcal{K}_2}) - f^* < \frac{4L_0R^2}{\mathcal{K}_2-\mathcal{K}_1}.
    \end{equation}

    \item[$(\mathcal{T}$)] Considering the scenario $c \geq \frac{L_0}{L_1} > \norms{\nabla f(x^k)}$ and summing up inequality \eqref{eq:bdhjbfjhbdjbvdfvhdjv} for $k=\mathcal{T}_2, \mathcal{T}_2+1, ..., \mathcal{T}_3-1$, we get
    \begin{equation*}
       \frac{\mathcal{T}_3-\mathcal{T}_2}{4  L_0 R^2} =  \sum_{k = \mathcal{T}_2}^{\mathcal{T}_3-1} \frac{1}{4  L_0 R^2}  <  \sum_{k = \mathcal{T}_2}^{\mathcal{T}_3-1} \left(\frac{1}{F_{k+1}} - \frac{1}{F_{k}}\right) = \frac{1}{F_{\mathcal{T}_3}} - \frac{1}{F_{\mathcal{T}_2}} < \frac{1}{F_{\mathcal{T}_3}},
    \end{equation*}
    which is equivalent to
    \begin{equation}\label{eq:clipGD2_scenari_1_2}
        f(x^{\mathcal{T}_3}) - f^* < \frac{4L_0R^2}{\mathcal{T}_3-\mathcal{T}_2}.
    \end{equation}
\end{itemize}
\end{enumerate}

Finally, combining \eqref{eq:clipGD1_scenari_1}, \eqref{eq:clipGD1_scenari_2}, \eqref{eq:clipGD2_scenari_1}, \eqref{eq:clipGD2_scenari_2}, \eqref{eq:clipGD2_scenari_1_2}, and taking into account that $F_{k+1} \leq F_k$, we obtain the convergence rate for Algorithm~\ref{algo:Clip-GD}.
\begin{itemize}
\begin{boxF}
    \item \fbox{If $c \geq \frac{L_0}{L_1}$,} then for $\mathcal{T}_3 = N$ being the total number of iterations of Algorithm~\ref{algo:Clip-GD} the iterates satisfy (see \eqref{eq:clipGD1_scenari_1}, \eqref{eq:clipGD2_scenari_1} and \eqref{eq:clipGD2_scenari_1_2}):
    \begin{equation*}
        f(x^N) - f^* = \mathcal{O} \left(\min\left\{\frac{L_0R^2}{N-\mathcal{T}_2}, \left( 1 - \frac{1}{L_1 R} \right)^{\mathcal{T}_2} F_0 \right\}\right),
    \end{equation*}
    where $\mathcal{T}_2 \coloneqq \min \left\{ k \in \{0,1,2,...,N-1\}\;\;\; | \;\;\;\norms{\nabla f(x^k)} < \frac{L_0}{L_1} \text{ and } \norms{\nabla f(x^{k-1})} \geq \frac{L_0}{L_1}\right\}$.
    \end{boxF}

\begin{boxF}
    \item \fbox{If $c < \frac{L_0}{L_1}$,} then $\mathcal{K}_2 = N$ being the total number of iterations of Algorithm~\ref{algo:Clip-GD} the iterates satisfy (see \eqref{eq:clipGD1_scenari_2} and \eqref{eq:clipGD2_scenari_2}):
    \begin{equation*}
        f(x^N) - f^* = \mathcal{O} \left(\min\left\{\frac{L_0R^2}{N-\mathcal{K}_1}, \left( 1 - \frac{c}{L_0 R} \right)^{\mathcal{K}_1} F_0 \right\}\right),
    \end{equation*}
    where $\mathcal{K}_1 = \coloneqq \min \left\{ k \in \{0,1,2,...,N-1\}\;\;\; | \;\;\;\norms{\nabla f(x^k)} < c \text{ and } \norms{\nabla f(x^{k-1})} \geq c\right\}$.
    \end{boxF}
\end{itemize}

It is not difficult to see that these two scenarios can be combined into the following equivalent form:
\begin{equation*}
           f(x^N) - f^* = \mathcal{O} \left(\min\left\{\frac{L_0R^2}{N-T}, \left( 1 - \frac{c}{\max\{L_0, L_1 c \}R} \right)^{T} F_0 \right\}\right),
        \end{equation*}
where $T \coloneqq \min \left\{ k \in \{0,1,2,...,N-1\}\;\;\; | \;\;\;\norms{\nabla f(x^k)} < \min\left\{\frac{L_0}{L_1}, c\right\} \text{ and } \norms{\nabla f(x^{k-1})} \geq \min\left\{\frac{L_0}{L_1}, c\right\}\right\}$.

\section{Missing Proofs for Coordinate Descent Type Methods}
In this section, we provide missing proofs from Section~\ref{sec:Coordinate Descent Type Methods}. In particular, see Subsection~\ref{subsec:proof_RCD} for the proof of the convergence results for Algorithm~\ref{algo:RCD}, and see Subsection~\ref{subsec:Proof_OrderRCD} for Algorithm~\ref{algo:OrderRCD}. 

\subsection{Proof of Theorem~\ref{th:RCD}}\label{subsec:proof_RCD}

Using Assumption~\ref{ass:L0_L1_coordinate_smooth}, we derive
\begin{align}
    f(x^{k+1}) - f(x^k) &= f(x^{k} - \eta_k \nabla_{i_k} f(x^k) \ee_{i_k}) - f(x^k) \nonumber \\
    &\overset{\eqref{eq:ass_smooth_coordinate}}{\leq} - \eta_k \left(\nabla_{i_k} f(x^k)\right)^2 + \eta_k^2 \frac{L_0 + L_1 |\nabla f(x^k)|}{2} \left(\nabla_{i_k} f(x^k)\right)^2 \nonumber \\
    & \overset{\circledOne}{\leq} - \eta_k \left(\nabla_{i_k} f(x^k)\right)^2 + \frac{\eta_k}{2} \left(\nabla_{i_k} f(x^k)\right)^2 \nonumber \\
    &= - \frac{\eta_k}{2} \left(\nabla_{i_k} f(x^k)\right)^2, \label{eq:RCD_non_increasing_F}
\end{align}
where in $\circledOne$ we used $\eta_k \leq \frac{1}{L_0 + L_1 |\nabla_{i_k}f(x^k)|}$. Next, we take the expectation w.r.t.\ $i_k$ and use $\eta_k = \frac{1}{L_0 + L_1 |\nabla_{i_k}f(x^k)|}$:
\begin{align}
    \E_{i_k}[f(x^{k+1})] - f(x^k) &\leq -\frac{1}{2d}\sum\limits_{i=1}^d \frac{|\nabla_{i}f(x^k)|^2}{L_0 + L_1 |\nabla_{i}f(x^k)|} \notag\\
    &\leq -\frac{1}{4d}\sum\limits_{i=1}^d \min\left\{\frac{|\nabla_{i}f(x^k)|^2}{L_0}, \frac{|\nabla_{i}f(x^k)|}{L_1}\right\}\\
    &= -\frac{1}{4d}\left(\sum\limits_{i\in I_k} \frac{|\nabla_{i}f(x^k)|}{L_1} + \sum\limits_{i\in [d]\setminus I_k} \frac{|\nabla_{i}f(x^k)|^2}{L_0}\right), \label{eq:RCD_1}
\end{align}
where $I_k \coloneqq \left\{i \in [d] \mid |\nabla_i f(x^k)| \geq \frac{L_0}{L_1} \right\}$. Next, we introduce the set of indices $\cK$ as $$\cK \coloneqq \left\{k \in [N-1] \mid  \sum\limits_{i\in I_k} |\nabla_{i}f(x^k)|^2 > \sum\limits_{i\in [d]\setminus I_k} |\nabla_{i}f(x^k)|^2\right\}$$ and consider two possible situations.

\fbox{The case of $k \in \cK$.} In this case, we continue our derivation as follows:
\begin{align}
    \E_{i_k}[f(x^{k+1})] - f(x^k) &\leq -\frac{1}{4dL_1}\sum\limits_{i\in I_k} |\nabla_{i}f(x^k)|. \label{eq:RCD_case1_1}
\end{align}
Using the convexity assumption and notation $F_k \coloneqq f(x^k) - f^*$, we derive
\begin{align*}
    F_k &\leq \dotprod{\nabla f(x^k)}{x^k - x^*}\overset{\eqref{eq:scalar_product_bound}}{\leq} \norms{\nabla f(x^k)} \underbrace{\norms{x^k - x^*}}_{R}\\
    &= R\sqrt{\sum\limits_{i\in I_k} |\nabla_{i}f(x^k)|^2 + \sum\limits_{i\in [d]\setminus I_k} |\nabla_{i}f(x^k)|^2} \leq R \sqrt{2\sum\limits_{i\in I_k} |\nabla_{i}f(x^k)|^2} \leq \sqrt{2}R\sum\limits_{i\in I_k} |\nabla_{i}f(x^k)|
\end{align*}
that implies
\begin{equation}
    \sum\limits_{i\in I_k} |\nabla_{i}f(x^k)| \geq \frac{F_k}{\sqrt{2}R}. \label{eq:RCD_case1_2}
\end{equation}
Plugging \eqref{eq:RCD_case1_2} in \eqref{eq:RCD_case1_1}, we obtain
\begin{equation}
    \E_{i_k}[F_{k+1}] \leq \left(1 - \frac{1}{4\sqrt{2}dL_1R}\right) F_k. \label{eq:RCD_case1_descent}
\end{equation}

\fbox{The case of $k \not\in \cK$.} In this case, we continue \eqref{eq:RCD_1} as follows:
\begin{align}
    \E_{i_k}[f(x^{k+1})] - f(x^k) &\leq -\frac{1}{4dL_0}\sum\limits_{i\in [d]\setminus I_k} |\nabla_{i}f(x^k)|^2. \label{eq:RCD_case2_1}
\end{align}
Using the convexity assumption and notation $F_k \coloneqq f(x^k) - f^*$, we derive
\begin{align*}
    F_k &\leq \dotprod{\nabla f(x^k)}{x^k - x^*}\overset{\eqref{eq:scalar_product_bound}}{\leq} \norms{\nabla f(x^k)} \underbrace{\norms{x^k - x^*}}_{R}\\
    &= R\sqrt{\sum\limits_{i\in I_k} |\nabla_{i}f(x^k)|^2 + \sum\limits_{i\in [d]\setminus I_k} |\nabla_{i}f(x^k)|^2} \leq R \sqrt{2\sum\limits_{i\in [d]\setminus I_k} |\nabla_{i}f(x^k)|^2}
\end{align*}
that implies
\begin{equation}
    \sum\limits_{i\in I_k} |\nabla_{i}f(x^k)|^2 \geq \frac{F_k^2}{2R^2}. \label{eq:RCD_case2_2}
\end{equation}
Plugging \eqref{eq:RCD_case2_2} in \eqref{eq:RCD_case2_1}, we obtain
\begin{equation}
    \E_{i_k}[F_{k+1}] \leq F_k  - \frac{1}{8dL_0 R^2}F_k^2. \label{eq:RCD_case2_descent}
\end{equation}

To get the final bound, let us specify the indices belonging to $\cK$: let $\cK \coloneqq \{k_1,k_2,\ldots,k_r\}$ and $\cT \coloneqq [N-1]\setminus \cK \coloneqq \{t_1, t_2, \ldots, t_{N-r}\}$, where $0 \leq k_1 \leq k_2 \leq \ldots \leq k_r \leq N-1$ and $0 \leq t_1 \leq t_2 \leq \ldots \leq t_{N-r} \leq N-1$. Note that $\cK \cap \cT = \varnothing$, $\cK \cup \cT = [N-1]$, and $|\cK| = r$ is random variable. There exist two possible situations: either $r > \nicefrac{N}{2}$ or $r \leq \nicefrac{N}{2}$. If $r > \nicefrac{N}{2}$, then we use \eqref{eq:RCD_case1_descent} together with $F_{k+1} \leq F_k$ following from \eqref{eq:RCD_non_increasing_F}:
\begin{align}
    \E_{i \in \cK}[F_{N}] &\overset{\eqref{eq:RCD_non_increasing_F}}{\leq}  \E_{i \in \cK}[F_{k_r + 1}] \overset{\eqref{eq:RCD_case1_descent}}{\leq} \left(1 - \frac{1}{4\sqrt{2}dL_1R}\right) \E_{i \in \cK\setminus\{k_r\}}[F_{k_r}] \notag \\
    &\overset{\eqref{eq:RCD_non_increasing_F}}{\leq} \left(1 - \frac{1}{4\sqrt{2}dL_1R}\right) \E_{i \in \cK\setminus\{k_r\}}[F_{k_{r-1}+1}] \overset{\eqref{eq:RCD_case1_descent}}{\leq} \left(1 - \frac{1}{4\sqrt{2}dL_1R}\right)^2 \E_{i \in \cK\setminus\{k_r, k_{r-1}\}}[F_{k_{r-1}}] \notag\\
    &\leq \ldots \leq \left(1 - \frac{1}{4\sqrt{2}dL_1R}\right)^r F_0 \overset{r > \nicefrac{N}{2}}{\leq} \left(1 - \frac{1}{4\sqrt{2}dL_1R}\right)^{\nicefrac{N}{2}} F_0 \mathbbm{1}_{\{r > \nicefrac{N}{2}\}}, \label{eq:RCD_final_case1}
\end{align}
where $\E_{i\in \cK}$ denotes the expectation w.r.t.\ all indices in set $\cK$ and $\mathbbm{1}_{\{r > \nicefrac{N}{2}\}}$ is an indicator of the event $\{r > \nicefrac{N}{2}\}$.

Next, we consider the situation when $r \leq \nicefrac{N}{2}$. In this case, we first notice that \eqref{eq:RCD_case2_descent} gives
\begin{align*}
    \E_{i \in \cT}[F_{t_{N-r}+1}] &\overset{\eqref{eq:RCD_case2_descent}}{\leq}  \E_{i \in \cT\setminus \{t_{N-r}\}}[F_{t_{N-r}}]  - \frac{1}{8dL_0 R^2}\E_{i \in \cT\setminus \{t_{N-r}\}}[F_{t_{N-r}}^2]\\
    &\overset{\circledOne}{\leq} \E_{i \in \cT\setminus \{t_{N-r}\}}[F_{t_{N-r}}]  - \frac{1}{8dL_0 R^2}\E_{i \in \cT\setminus \{t_{N-r}\}}[F_{t_{N-r}}]^2
\end{align*}
where in $\circledOne$ we used $\E_{i \in \cT\setminus \{t_{N-r}\}}[F_{t_{N-r}}]^2 \leq \E_{i \in \cT\setminus \{t_{N-r}\}}[F_{t_{N-r}}^2]$. Dividing both sides by\newline $\E_{i \in \cT}[F_{t_{N-r}+1}]\E_{i \in \cT\setminus \{t_{N-r}\}}[F_{t_{N-r}}]$ and rearranging the terms, we get
\begin{equation}
    \frac{1}{8dL_0 R^2}\frac{\E_{i \in \cT\setminus \{t_{N-r}\}}[F_{t_{N-r}}]}{\E_{i \in \cT}[F_{t_{N-r}+1}]} \leq \frac{1}{\E_{i \in \cT}[F_{t_{N-r}+1}]} - \frac{1}{\E_{i \in \cT\setminus \{t_{N-r}\}}[F_{t_{N-r}}]}. \label{eq:hbjsdbvhjdf}
\end{equation}
In view of \eqref{eq:RCD_non_increasing_F}, we have $\E_{i \in \cT}[F_{t_{N-r}+1}] \leq \E_{i \in \cT}[F_{t_{N-r}}] = \E_{i \in \cT\setminus \{t_{N-r}\}}[F_{t_{N-r}}]$ and $- \frac{1}{\E_{i \in \cT\setminus \{t_{N-r}\}}[F_{t_{N-r}}]} \leq - \frac{1}{\E_{i \in \cT\setminus \{t_{N-r}\}}[F_{t_{N-r-1}+1}]} = - \frac{1}{\E_{i \in \cT}[F_{t_{N-r-1}+1}]}$. Using these inequalities in \eqref{eq:hbjsdbvhjdf}, we obtain
\begin{equation*}
    \frac{1}{8dL_0 R^2} \leq \frac{1}{\E_{i \in \cT}[F_{t_{N-r}+1}]} - \frac{1}{\E_{i \in \cT}[F_{t_{N-r-1}+1}]}. 
\end{equation*}
Following the same arguments, we can also show
\begin{align*}
    \frac{1}{8dL_0 R^2} &\leq \frac{1}{\E_{i \in \cT}[F_{t_{N-r-1}+1}]} - \frac{1}{\E_{i \in \cT}[F_{t_{N-r-2}+1}]},\\
    &\ldots\\
    \frac{1}{8dL_0 R^2} &\leq \frac{1}{\E_{i \in \cT}[F_{t_{1}+1}]} - \frac{1}{\E_{i \in \cT}[F_{t_{1}}]} \leq \frac{1}{\E_{i \in \cT}[F_{t_{1}+1}]} - \frac{1}{F_0}.
\end{align*}
Summing up all of them, we arrive at
\begin{equation*}
    \frac{N-r}{8dL_0 R^2} \leq \frac{1}{\E_{i \in \cT}[F_{t_{N-r}+1}]} - \frac{1}{F_0} \leq \frac{1}{\E_{i \in \cT}[F_{t_{N-r}+1}]},
\end{equation*}
implying
\begin{equation}
    \E_{i \in \cT}[F_{N}] \leq \E_{i \in \cT}[F_{t_{N-r}+1}] \leq \frac{8dL_0 R^2}{N-r} \overset{r \leq \nicefrac{N}{2}}{\leq} \frac{16dL_0 R^2}{N} \mathbbm{1}_{\{r \leq \nicefrac{N}{2}\}}. \label{eq:RCD_final_case2}
\end{equation}

Combining \eqref{eq:RCD_final_case1} and \eqref{eq:RCD_final_case2} and taking the full expectation, we get
\begin{align*}
    \E[F_{N}] &\leq \left(1 - \frac{1}{4\sqrt{2}dL_1R}\right)^{\nicefrac{N}{2}} F_0 \E[\mathbbm{1}_{\{r > \nicefrac{N}{2}\}}] + \frac{16dL_0 R^2}{N} \E[\mathbbm{1}_{\{r \leq \nicefrac{N}{2}\}}]\\
    &\leq \max\left\{\left(1 - \frac{1}{4\sqrt{2}dL_1R}\right)^{\nicefrac{N}{2}} F_0, \frac{16dL_0 R^2}{N}\right\},
\end{align*}
which concludes the proof.

\subsection{Proof of Theorem~\ref{th:OrderRCD}}\label{subsec:Proof_OrderRCD}
Algorithm \ref{algo:OrderRCD}, presented in Section~\ref{sec:Coordinate Descent Type Methods}, uses the Golden Ratio Method (GRM) once per iteration. This method utilizes the oracle concept \eqref{eq:Order_Oracle} (see Algorithm \ref{algo:GRM}).

\begin{algorithm}[H]
    \caption{Golden Ratio Method (GRM)}
    \label{algo:GRM}
    \begin{algorithmic}[1]
    \State {\bfseries Input:} interval $[a, b]$, accuracy $\hat\epsilon$
    \State {\bfseries Initialization:}  define constant $\rho = \frac{1}{\Phi} = \frac{\sqrt{5} - 1}{2}$
    \State $y \gets a + (1 - \rho)(b - a)$
    \State $z \gets a + \rho(b - a)$
    \While{$b - a > \hat\epsilon$}
        \If{$\phi(y,z) = -1$}
            \State $b \gets z$
            \State $z \gets y$
            \State $y \gets a + (1 - \rho)(b - a)$
        \Else
            \State $a \gets y$
            \State  $y \gets z$
            \State $z \gets a + \rho(b - a)$
        \EndIf
    \EndWhile
   \State {\bfseries Return:} $\frac{a + b}{2}$
    \end{algorithmic}
\end{algorithm}

We utilize the Golden Ratio Method to find a solution to the following one-dimensional problem (see line 2 in Algorithm~\ref{algo:OrderRCD}):

\begin{equation*}
    \zeta_k = \argmin_{\zeta \in \mathbb{R}} f(x^k + \zeta \ee_{i_k}).
\end{equation*}

Using the well-known fact about the golden ratio method that GRM is required to do $N = \Obound{\log \frac{1}{\epsilon}}$ (where $\epsilon$ is the accuracy of the solution to the linear search problem in terms of the function value; due to \eqref{eq:ass_smooth_coordinate}, it is sufficient to take $\hat{\epsilon} = \nicefrac{2\epsilon}{L_0}$), we derive the following corollaries from the solution of this problem: for simplicity, we consider the scenario when the golden ratio method solves the inner problem exactly ($\epsilon \simeq 0$). Then, we have the following:
    \begin{equation}
    \label{eq:col_without_noise}
        f(x_k + \zeta_k \ee_{i_k}) \leq f(x_k + \zeta \ee_{i_k}), \quad \quad \quad \forall \zeta \in \mathbb{R}.
    \end{equation}
Using the above inequality with $\zeta = \eta_k \nabla_{i_k}f(x^k)$, $\eta_k \coloneqq \frac{1}{L_0 + L_1|\nabla_{i_k}f(x^k)|}$, and applying Assumption~\ref{ass:L0_L1_coordinate_smooth}, we get
\begin{align}
    f(x^{k+1}) - f(x^k) &= f(x^k + \zeta_k \ee_{i_k}) - f(x^k)\nonumber\\
    &\overset{\eqref{eq:col_without_noise}}{\leq} f(x^{k} - \eta_k \nabla_{i_k} f(x^k) \ee_{i_k}) - f(x^k) \nonumber \\
    &\overset{\eqref{eq:RCD_non_increasing_F}}{\leq}  - \frac{\eta_k}{2} \left(\nabla_{i_k} f(x^k)\right)^2. \label{eq:OrderRCD_1}
\end{align}
The rest of the proof is identical to the proof given in Appendix~\ref{subsec:proof_RCD} and leads to the same bound:
\begin{equation*}
    \E[f(x^N)] - f^* \leq \max\left\{\left(1 - \frac{1}{4\sqrt{2}dL_1R}\right)^{\nicefrac{N}{2}} F_0, \frac{16dL_0 R^2}{N}\right\}.
\end{equation*}
The above upper bound implies that to achieve $\E[f(x^N)] - f^* \leq \varepsilon$, OrderRCD needs to perform 
\begin{equation*}
    N = \mathcal{O}\left( \max\left\{ \frac{dL_0R^2}{\varepsilon}, dL_1R \log \frac{F_0}{\varepsilon} \right\} \right)\;\;\;\;\;\; \text{iterations and}
\end{equation*}

\begin{equation*}
    T = \mathcal{O}\left( \max\left\{ \frac{dL_0R^2}{\varepsilon}, dL_1R \log \frac{F_0}{\varepsilon} \right\} \cdot \log \frac{1}{\epsilon} \right)\;\;\;\;\;\; \text{Order Oracle calls,}
\end{equation*}
where $\epsilon$ is the accuracy of the solution of the auxiliary optimization problem (see line 2 in Algorithm~\ref{algo:OrderRCD}), and it has to be sufficiently small.

\section{Missing Proof for GD in the Strongly Convex Setup}\label{app:Proof_GD_strongly}

From the analysis of the convex case, we have
\begin{align}
    f(x^{k+1}) - f(x^k) &\overset{\eqref{eq:GD_convex_1}}{\leq} - \frac{1}{2(L_0 + L_1 \norms{\nabla f(x^k)})}\norms{\nabla f(x^k)}^2. \label{eq:GD_strongly_1}
\end{align}
Next, let us consider two cases: $\norms{\nabla f(x^k)} \geq \frac{L_0}{L_1}$ and $\norms{\nabla f(x^k)} < \frac{L_0}{L_1}$.

\fbox{The case of $\norms{\nabla f(x^k)} \geq \frac{L_0}{L_1}$.} In this case, we have $L_0 + L_1 \norms{\nabla f(x^k)} \leq 2L_1 \norms{\nabla f(x^k)}$. Using this inequality in \eqref{eq:GD_strongly_1}, we obtain
\begin{equation}
    f(x^{k+1}) - f(x^k) \leq - \frac{\|\nabla f(x^k)\|}{4L_1}. \label{eq:GD_strongly_1_1}
\end{equation}
To continue the derivation, we also consider two possible situations depending on $F_k \coloneqq f(x^k) - f^*$.
\begin{enumerate}
    \item[i)] If $F_k \geq 1$, then we proceed as in the convex case and get 
    \begin{equation}
        F_{k+1} \overset{\eqref{eq:bdvnjkdfkvf}}{\leq} \left( 1 - \frac{1}{4 L_1 R} \right) F_k. \label{eq:bdvnjkdfkvf_str_cvx}
    \end{equation}
    In view of \eqref{eq:GD_strongly_1} and Lemma~\ref{lem:GD}, we have $F_{k+1} \leq F_k$ and $\|\nabla f(x^{k+1})\| \leq \|\nabla f(x^k)\|$. Therefore, if $F_k \geq 1$ and $\norms{\nabla f(x^k)} \geq \frac{L_0}{L_1}$, then $F_t \geq 1$ and $\norms{\nabla f(x^t)} \geq \frac{L_0}{L_1}$ for all $t = 0,1,\ldots, k$. Let $\cT_1$ be the largest $k \in [N-1]$ such that $F_k \geq 1$ and $\norms{\nabla f(x^k)} \geq \frac{L_0}{L_1}$ (if there is no such $k$ for given initialization, then $\cT_1 \coloneqq -1$). Then, we have
    \begin{equation}
        F_{\cT_1+1} \overset{\eqref{eq:bdvnjkdfkvf_str_cvx}}{\leq} \left( 1 - \frac{1}{4 L_1 R} \right)^{\cT_1+1} F_0. \label{eq:bdvnjkdfkvf_str_cvx_2}
    \end{equation}
    Using the above inequality, we can upper bound $\cT_1$ as
    \begin{align*}
        \mathcal{T}_1 \leq 4L_1R \log (F_0). 
    \end{align*}

    \item[ii)] If $F_k < 1$, we use Polyak-{\L}ojasiewicz \citep{polyak1963gradient, lojasiewicz1963topological} inequality
    \begin{align}
        \norms{\nabla f(x^k)}^2 &\geq 2\mu F_k \label{eq:PL_ineq} \\
        &\overset{F_k < 1}{>} 2\mu (F_k)^2, \label{eq:GD_strongly1_case2_1}
    \end{align}
    which follows from strong convexity \cite{Nesterov_2018}. Then, we can continue the derivation as follows:
    \begin{align}
        F_{k+1} &\overset{\eqref{eq:GD_strongly_1_1}}{\leq} F_k  - \frac{1}{4L_1} \norms{\nabla f(x^k)} \nonumber\\
        &\overset{\eqref{eq:GD_strongly1_case2_1}}{\leq} \left(1 - \frac{\sqrt{\mu}}{2\sqrt{2}L_1}\right) F_k. \label{eq:GD_strongly1_case3_2}
    \end{align}
    Moreover, since \eqref{eq:bdvnjkdfkvf_str_cvx} holds whenever $\norms{\nabla f(x^k)} \geq \frac{L_0}{L_1}$, we can tighten the above inequality as
    \begin{align}
        F_{k+1} &\leq \left(1 - \max\left\{\frac{\sqrt{\mu}}{2\sqrt{2}L_1}, \frac{1}{4L_1R}\right\}\right) F_k. \label{eq:GD_strongly1_case3_3}
    \end{align}
    In view of Lemma~\ref{lem:GD}, we have $\|\nabla f(x^{k+1})\| \leq \|\nabla f(x^k)\|$. Therefore, if $\norms{\nabla f(x^k)} \geq \frac{L_0}{L_1}$, then $\norms{\nabla f(x^t)} \geq \frac{L_0}{L_1}$ for all $t = 0,1,\ldots, k$. Let $\cT_2$ be the largest $k \in [N-1]$ such that $\norms{\nabla f(x^k)} \geq \frac{L_0}{L_1}$ (if there is no such $k$ for given initialization, then $\cT_2 \coloneqq -1$). Then, we have
    \begin{align}
        F_{\mathcal{T}_2+1} &\leq \left(1 - \max\left\{\frac{\sqrt{\mu}}{2\sqrt{2}L_1}, \frac{1}{4L_1R}\right\}\right)^{\mathcal{T}_2 - \mathcal{T}_1} F_{\mathcal{T}_1+1} \notag\\
        &\overset{\eqref{eq:bdvnjkdfkvf_str_cvx_2}}{\leq} \left(1 - \max\left\{\frac{\sqrt{\mu}}{2\sqrt{2}L_1}, \frac{1}{4L_1R}\right\}\right)^{\mathcal{T}_2 - \mathcal{T}_1}\left( 1 - \frac{1}{4 L_1 R} \right)^{\cT_1+1} F_0.\label{eq:GD_strongly_scenari_3}
    \end{align}
\end{enumerate}

\fbox{The case of $\norms{\nabla f(x^k)} < \frac{L_0}{L_1}$.} In this case, we have $L_0 + L_1\|\nabla f(x^k)\| \leq 2L_0$. Using this inequality in \eqref{eq:GD_strongly_1}, we obtain
    \begin{align}
        F_{k+1} &\overset{\eqref{eq:GD_strongly_1}}{\leq} F_k - \frac{\norms{\nabla f(x^k)}^2}{4L_0} \nonumber\\
        &\overset{\eqref{eq:PL_ineq}}{\leq} \left(1 - \frac{\mu}{2L_0}\right) F_k. \label{eq:GD_strongly2_case1_2}
    \end{align}
    Since Algorithm~\ref{algo:GD} converges monotonically in terms of the gradient norm (see Appendix~\ref{app:Monotonicity of Algorithms by Gradient Norms}), the above inequality holds for $k=\mathcal{T}_2+1, \mathcal{T}_2+2\ldots,N-1$ iterations and gives
    \begin{align}
        F_N &\leq \left( 1 - \frac{\mu}{2L_0} \right)^{N-\cT_2} F_{\cT_2+1}\notag \\ &\overset{\eqref{eq:GD_strongly_scenari_3}}{\leq} \left( 1 - \frac{\mu}{2L_0} \right)^{N-\cT_2}\left(1 - \max\left\{\frac{\sqrt{\mu}}{2\sqrt{2}L_1}, \frac{1}{4L_1R}\right\}\right)^{\mathcal{T}_2 - \mathcal{T}_1}\left( 1 - \frac{1}{4 L_1 R} \right)^{\cT_1+1} F_0.\notag
    \end{align}
This concludes the proof.

\section{Motivation Strong Growth Conditions on the Example of Logistic Regression}

We know from Section~\ref{sec:Discussion and Future Work} that the strong growth condition for smoothness (see Assumption~\ref{ass:L0_L1_smooth} when $L_0 = 0$) is satisfied by the logistic regression problem, which is often used in the machine learning community. However, this problem does not reach a minimum (hence $R = \arginf f(x) = + \infty$). Therefore, in this section we show that, for example, gradient descent (Algorithm~\ref{algo:GD}) will achieve the desired accuracy $\varepsilon$ in a finite number of iterations with linear rate. 

We introduce the hyperparameter of the algorithm $s: f(s) - f^* \leq \varepsilon$. Then we show linear convergence to the desired accuracy by the example of gradient descent.

Using the strong growth condition for smoothness (see Assumption~\ref{ass:L0_L1_smooth} with $L_0 = 0$) we have:
\begin{align}
    f(x^{k+1}) - f(x^k) &= f(x^{k} - \eta_k \nabla f(x^k)) - f(x^k) \nonumber \\
    &\overset{\eqref{eq:ass_smooth}}{\leq} - \eta_k \dotprod{\nabla f(x^k)}{\nabla f(x^k)} + \eta_k^2 \frac{L_1 \norms{\nabla f(x^k)}}{2} \norms{\nabla f(x^k)}^2 \nonumber \\
    & \overset{\circledOne}{\leq} - \eta_k \norms{\nabla f(x^k)}^2 + \frac{\eta_k}{2} \norms{\nabla f(x^k)}^2 \nonumber \\
    &= - \frac{\eta_k}{2} \norms{\nabla f(x^k)}^2 \nonumber \\
    & = - \frac{1}{2 L_1 \norms{\nabla f(x^k)}} \norms{\nabla f(x^k)}^2 \nonumber\\
    & = - \frac{1}{2 L_1} \norms{\nabla f(x^k)}, \label{eq:SGC_GD_convex_1}
\end{align}
where in $\circledOne$ we used $\eta_k \leq \frac{1}{L_1 \norms{\nabla f(x^k)}}$.

Then using the convexity assumption of the function (see Assumption~\ref{ass:strongly_convex}, $\mu = 0$), we have the following:
\begin{align*}
    f(x^k) - f(s) &\leq \dotprod{\nabla f(x^k)}{x^k - s}\\
    &\overset{\eqref{eq:scalar_product_bound}}{\leq} \norms{\nabla f(x^k)} \norms{x^k - s}\\
    & \leq \norms{\nabla f(x^k)} \underbrace{\norms{x^0 - s}}_{R_s}.
\end{align*}
Hence we have:
\begin{equation}
    \label{eq:SGC_GD_convex_2}
    \norms{\nabla f(x^k)} \geq \frac{f(x^k) - f(s)}{R_s}.
\end{equation}
Then substituting \eqref{eq:SGC_GD_convex_2} into \eqref{eq:SGC_GD_convex_1} we obtain:
\begin{equation*}
    f(x^{k+1}) - f(x^k) \leq - \frac{1}{2 L_1} \norms{\nabla f(x^k)} \leq - \frac{1}{2 L_1 R_s} (f(x^k) - f(s)).
\end{equation*}

This inequality is equivalent to the trailing inequality:
\begin{equation}
    f(x^{k+1}) - f^* \leq \left( 1 - \frac{1}{2 L_1 R_s}\right) (f(x^k) - f^*) + \frac{1}{2 L_1 R_s} (f(s) - f^*). \label{eq:SGC_GD_convex_3}
\end{equation}

Applying recursion to \eqref{eq:SGC_GD_convex_3} we obtain:
\begin{equation*}
    f(x^{N}) - f^*  \leq \left( 1 - \frac{1}{2 L_1 R_s}\right)^N (f(x^0) - f^*) + \underbrace{(f(s) - f^*)}_{\varepsilon}.
\end{equation*}

Therefore, we have shown that Algorithm~\ref{algo:GD} will achieve the desired accuracy $\varepsilon$ in a finite number of iterations: ${N = \mathcal{O} \left( L_1 R_s \log \frac{1}{\varepsilon} \right)}$. $R_s$ is a finite number and increases as the desired accuracy improves. The same can be shown for other algorithms.

%% file: 0_article.bbl
\begin{thebibliography}{}

\bibitem[Allen-Zhu et~al., 2016]{Allen_2016}
Allen-Zhu, Z., Qu, Z., Richt{\'a}rik, P., and Yuan, Y. (2016).
\newblock Even faster accelerated coordinate descent using non-uniform sampling.
\newblock In {\em International Conference on Machine Learning}, pages 1110--1119. PMLR.

\bibitem[Bai et~al., 2022]{bai2022training}
Bai, Y., Jones, A., Ndousse, K., Askell, A., Chen, A., DasSarma, N., Drain, D., Fort, S., Ganguli, D., Henighan, T., et~al. (2022).
\newblock Training a helpful and harmless assistant with reinforcement learning from human feedback.
\newblock {\em arXiv preprint arXiv:2204.05862}.

\bibitem[Boyd and Vandenberghe, 2004]{Boyd_2004}
Boyd, S. and Vandenberghe, L. (2004).
\newblock {\em Convex optimization}.
\newblock Cambridge university press.

\bibitem[Bubeck et~al., 2015]{Bubeck_2015}
Bubeck, S. et~al. (2015).
\newblock Convex optimization: Algorithms and complexity.
\newblock {\em Foundations and Trends{\textregistered} in Machine Learning}, 8(3-4):231--357.

\bibitem[Cauchy, 1847]{Cauchy_1847}
Cauchy, A. (1847).
\newblock M{\'e}thode g{\'e}n{\'e}rale pour la r{\'e}solution des systemes d’{\'e}quations simultan{\'e}es.
\newblock {\em Comp. Rend. Sci. Paris}, 25(1847):536--538.

\bibitem[Chen et~al., 2023]{Chen_2023}
Chen, Z., Zhou, Y., Liang, Y., and Lu, Z. (2023).
\newblock Generalized-smooth nonconvex optimization is as efficient as smooth nonconvex optimization.
\newblock In {\em International Conference on Machine Learning}, pages 5396--5427. PMLR.

\bibitem[Crawshaw et~al., 2022]{crawshaw2022robustness}
Crawshaw, M., Liu, M., Orabona, F., Zhang, W., and Zhuang, Z. (2022).
\newblock Robustness to unbounded smoothness of generalized signsgd.
\newblock {\em Advances in neural information processing systems}, 35:9955--9968.

\bibitem[Faw et~al., 2023]{faw2023beyond}
Faw, M., Rout, L., Caramanis, C., and Shakkottai, S. (2023).
\newblock Beyond uniform smoothness: A stopped analysis of adaptive sgd.
\newblock In {\em The Thirty Sixth Annual Conference on Learning Theory}, pages 89--160. PMLR.

\bibitem[Gorbunov et~al., 2019]{Gorbunov_2019}
Gorbunov, E., Bibi, A., Sener, O., Bergou, E.~H., and Richtarik, P. (2019).
\newblock A stochastic derivative free optimization method with momentum.
\newblock In {\em International Conference on Learning Representations}.

\bibitem[Gorbunov et~al., 2024]{Gorbunov_2024}
Gorbunov, E., Tupitsa, N., Choudhury, S., Aliev, A., Richt{\'a}rik, P., Horv{\'a}th, S., and Tak{\'a}{\v{c}}, M. (2024).
\newblock Methods for convex $(l\_0, l\_1) $-smooth optimization: Clipping, acceleration, and adaptivity.
\newblock {\em arXiv preprint arXiv:2409.14989}.

\bibitem[H{\"u}bler et~al., 2024]{hubler2024parameter}
H{\"u}bler, F., Yang, J., Li, X., and He, N. (2024).
\newblock Parameter-agnostic optimization under relaxed smoothness.
\newblock In {\em International Conference on Artificial Intelligence and Statistics}, pages 4861--4869. PMLR.

\bibitem[Koloskova et~al., 2023]{Koloskova_2023}
Koloskova, A., Hendrikx, H., and Stich, S.~U. (2023).
\newblock Revisiting gradient clipping: Stochastic bias and tight convergence guarantees.
\newblock In {\em International Conference on Machine Learning}, pages 17343--17363. PMLR.

\bibitem[Li et~al., 2024a]{li2024convex}
Li, H., Qian, J., Tian, Y., Rakhlin, A., and Jadbabaie, A. (2024a).
\newblock Convex and non-convex optimization under generalized smoothness.
\newblock {\em Advances in Neural Information Processing Systems}, 36.

\bibitem[Li et~al., 2024b]{li2024convergence}
Li, H., Rakhlin, A., and Jadbabaie, A. (2024b).
\newblock Convergence of adam under relaxed assumptions.
\newblock {\em Advances in Neural Information Processing Systems}, 36.

\bibitem[Lin et~al., 2014]{Lin_2014}
Lin, Q., Lu, Z., and Xiao, L. (2014).
\newblock An accelerated proximal coordinate gradient method.
\newblock {\em Advances in Neural Information Processing Systems}, 27.

\bibitem[Lobanov et~al., 2024]{Lobanov_2024}
Lobanov, A., Gasnikov, A., and Krasnov, A. (2024).
\newblock Acceleration exists! optimization problems when oracle can only compare objective function values.
\newblock In {\em The Thirty-eighth Annual Conference on Neural Information Processing Systems}.

\bibitem[{\L}ojasiewicz, 1963]{lojasiewicz1963topological}
{\L}ojasiewicz, S. (1963).
\newblock A topological property of real analytic subsets.
\newblock {\em Coll. du CNRS, Les {\'e}quations aux d{\'e}riv{\'e}es partielles}, 117:87--89.

\bibitem[Nesterov, 2012]{Nesterov_2012}
Nesterov, Y. (2012).
\newblock Efficiency of coordinate descent methods on huge-scale optimization problems.
\newblock {\em SIAM Journal on Optimization}, 22(2):341--362.

\bibitem[Nesterov, 2013]{Nesterov_2013}
Nesterov, Y. (2013).
\newblock {\em Introductory lectures on convex optimization: A basic course}, volume~87.
\newblock Springer Science \& Business Media.

\bibitem[Nesterov, 2018]{Nesterov_2018}
Nesterov, Y. (2018).
\newblock {\em Lectures on convex optimization}, volume 137.
\newblock Springer.

\bibitem[Ouyang et~al., 2022]{ouyang2022training}
Ouyang, L., Wu, J., Jiang, X., Almeida, D., Wainwright, C., Mishkin, P., Zhang, C., Agarwal, S., Slama, K., Ray, A., et~al. (2022).
\newblock Training language models to follow instructions with human feedback.
\newblock {\em Advances in neural information processing systems}, 35:27730--27744.

\bibitem[Pascanu et~al., 2013]{Pascanu_2013}
Pascanu, R., Mikolov, T., and Bengio, Y. (2013).
\newblock On the difficulty of training recurrent neural networks.
\newblock In Dasgupta, S. and McAllester, D., editors, {\em Proceedings of the 30th International Conference on Machine Learning}, volume~28 of {\em Proceedings of Machine Learning Research}, pages 1310--1318, Atlanta, Georgia, USA. PMLR.

\bibitem[Polyak, 1963]{polyak1963gradient}
Polyak, B.~T. (1963).
\newblock Gradient methods for the minimisation of functionals.
\newblock {\em USSR Computational Mathematics and Mathematical Physics}, 3(4):864--878.

\bibitem[Richt{\'a}rik and Tak{\'a}{\v{c}}, 2016]{Richtarik_2016}
Richt{\'a}rik, P. and Tak{\'a}{\v{c}}, M. (2016).
\newblock Distributed coordinate descent method for learning with big data.
\newblock {\em Journal of Machine Learning Research}, 17(75):1--25.

\bibitem[Saha et~al., 2021]{Saha_2021}
Saha, A., Koren, T., and Mansour, Y. (2021).
\newblock Dueling convex optimization.
\newblock In {\em International Conference on Machine Learning}, pages 9245--9254. PMLR.

\bibitem[Shalev-Shwartz and Tewari, 2009]{Shalev_2009}
Shalev-Shwartz, S. and Tewari, A. (2009).
\newblock Stochastic methods for l 1 regularized loss minimization.
\newblock In {\em Proceedings of the 26th Annual International Conference on Machine Learning}, pages 929--936.

\bibitem[Takezawa et~al., 2024]{Takezawa_2024}
Takezawa, Y., Bao, H., Sato, R., Niwa, K., and Yamada, M. (2024).
\newblock Polyak meets parameter-free clipped gradient descent.
\newblock {\em arXiv preprint arXiv:2405.15010}.

\bibitem[Tang et~al., 2024]{Tang_2024}
Tang, Z., Rybin, D., and Chang, T.-H. (2024).
\newblock Zeroth-order optimization meets human feedback: Provable learning via ranking oracles.
\newblock In {\em The Twelfth International Conference on Learning Representations}.

\bibitem[Vankov et~al., 2024a]{Vankov_2024_ver_2}
Vankov, D., Rodomanov, A., Nedich, A., Sankar, L., and Stich, S.~U. (2024a).
\newblock Optimizing $({L}_0, {L}_1)$-smooth functions by gradient methods.
\newblock {\em arXiv preprint arXiv:2410.10800, version 2}.

\bibitem[Vankov et~al., 2024b]{Vankov_2024}
Vankov, D., Rodomanov, A., Nedich, A., Sankar, L., and Stich, S.~U. (2024b).
\newblock Optimizing $({L}_0, {L}_1)$-smooth functions by gradient methods.
\newblock {\em arXiv preprint arXiv:2410.10800, version 1}.

\bibitem[Vaswani et~al., 2019]{Vaswani_2019}
Vaswani, S., Bach, F., and Schmidt, M. (2019).
\newblock Fast and faster convergence of sgd for over-parameterized models and an accelerated perceptron.
\newblock In {\em The 22nd international conference on artificial intelligence and statistics}, pages 1195--1204. PMLR.

\bibitem[Wang et~al., 2023]{wang2023convergence}
Wang, B., Zhang, H., Ma, Z., and Chen, W. (2023).
\newblock Convergence of adagrad for non-convex objectives: Simple proofs and relaxed assumptions.
\newblock In {\em The Thirty Sixth Annual Conference on Learning Theory}, pages 161--190. PMLR.

\bibitem[Wang et~al., 2022]{wang2022provable}
Wang, B., Zhang, Y., Zhang, H., Meng, Q., Ma, Z.-M., Liu, T.-Y., and Chen, W. (2022).
\newblock Provable adaptivity in adam.
\newblock {\em arXiv preprint arXiv:2208.09900}.

\bibitem[Zhang et~al., 2020a]{Zhang_2020}
Zhang, B., Jin, J., Fang, C., and Wang, L. (2020a).
\newblock Improved analysis of clipping algorithms for non-convex optimization.
\newblock {\em Advances in Neural Information Processing Systems}, 33:15511--15521.

\bibitem[Zhang et~al., 2020b]{Zhang_2019}
Zhang, J., He, T., Sra, S., and Jadbabaie, A. (2020b).
\newblock Why gradient clipping accelerates training: A theoretical justification for adaptivity.
\newblock In {\em International Conference on Learning Representations}.

\bibitem[Zhang and Xiao, 2017]{Zhang_2017}
Zhang, Y. and Xiao, L. (2017).
\newblock Stochastic primal-dual coordinate method for regularized empirical risk minimization.
\newblock {\em Journal of Machine Learning Research}, 18(84):1--42.

\bibitem[Zhao et~al., 2021]{zhao2021convergence}
Zhao, S.-Y., Xie, Y.-P., and Li, W.-J. (2021).
\newblock On the convergence and improvement of stochastic normalized gradient descent.
\newblock {\em Science China Information Sciences}, 64:1--13.

\end{thebibliography}
